\documentclass[11pt,reqno]{amsart}
\usepackage{amsmath,amsxtra,latexsym,amsthm,amssymb,amscd,pb-diagram}
\usepackage{mathrsfs,mathabx,amsfonts}
\usepackage[colorlinks=true,linkcolor=magenta,citecolor=blue]{hyperref}
\usepackage[margin=1in]{geometry}

\usepackage{bm}
\usepackage{color}
\usepackage{makecell,multirow,diagbox}
\usepackage{mathtools}
\usepackage[displaymath, mathlines]{lineno}

\usepackage{graphicx}
\usepackage{epsfig}
\usepackage{array}
\usepackage{enumitem}

\newtheorem{theorem}{Theorem}[section]
\newtheorem{proposition}{Proposition}[section]
\newtheorem{lemma}{Lemma}[section]

\newtheorem{remark}{Remark}[section]

\newcommand{\R}{\mathbb{R}}

\newcommand{\ity}{\infty}

\newcommand{\f}{\displaystyle\frac}

\begin{document}
\title[Asymptotic properties of $\sigma$-evolution equations with general double damping]{On asymptotic properties of solutions to $\sigma$-evolution equations with general double damping}

\subjclass{35B40, 35B44, 35L30, 35L56}
\keywords{$\sigma$-evolution equations, Double damping, Asymptotic profiles, Global solutions}
\thanks{$^* $\textit{Corresponding author:} Tuan Anh Dao (anh.daotuan@hust.edu.vn)}

\maketitle
\centerline{\scshape Tuan Anh Dao$^{1,*}$, Dinh Van Duong$^1$, Duc Anh Nguyen$^1$}
\medskip
{\footnotesize
	\centerline{$^1$ Faculty of Mathematics and Informatics, Hanoi University of Science and Technology}
	\centerline{No.1 Dai Co Viet road, Hanoi, Vietnam}}

\begin{abstract}
In this paper, we would like to consider the Cauchy problem for semi-linear $\sigma$-evolution equations with double structural damping for any $\sigma\ge 1$. The main purpose of the present work is to not only study the asymptotic profiles of solutions to the corresponding linear equations but also describe large-time behaviors of globally obtained solutions to the semi-linear equations. We want to emphasize that the new contribution is to find out the sharp interplay of ``parabolic like models" corresponding to $\sigma_1 \in [0,\sigma/2)$ and ``$\sigma$-evolution like models" corresponding to $\sigma_2 \in (\sigma/2,\sigma]$, which together appear in an equation. In this connection, we understand clearly how each damping term influences the asymptotic properties of solutions.
\end{abstract}

% \linenumbers
\tableofcontents

%====================================================================================
%=================================================================================={Introduction}	
\section{Introduction}
\subsection{Background of the damped $\sigma$-evolution equations}
First, let us mention some historical views in terms of studying the following Cauchy problem for damped $\sigma$-evolution equations:
\begin{equation}\label{Initial.Eq}
    \begin{cases}
u_{tt}+ (-\Delta)^\sigma u+ \mu(-\Delta)^{\delta} u_t= |\partial_t^j u|^p, &\quad x\in \R^n,\, t \ge 0, \\
u(0,x)= u_0(x),\quad u_t(0,x)= u_1(x), &\quad x\in \R^n,
\end{cases}
\end{equation}
where $\sigma \ge 1$, $\delta \in [0,\sigma]$ and $\mu$ is a positive constant. The term $\mu(-\Delta)^{\delta} u_t$ represents a damping term, more precisely, is often called a frictional (or external) damping, a structural damping, or a visco-elastic (or strong) damping when $\delta=0$, $\delta\in (0,\sigma)$ or $\delta=\sigma$, respectively. The parameter $p>1$ stands for power exponents of the nonlinear terms. Here, the two cases of $j=0$ (the usual power nonlinearity) and $j=1$ (the power nonlinearity of derivative type in time) are of main interest. \medskip

To get started, we recall several results for the corresponding linear equation of \eqref{Initial.Eq} in the form
\begin{equation}\label{Linear_Initial.Eq}
    \begin{cases}
u_{tt}+ (-\Delta)^\sigma u+ \mu(-\Delta)^{\delta} u_t= 0, &\quad x\in \R^n,\, t \ge 0, \\
u(0,x)= u_0(x),\quad u_t(0,x)= u_1(x), &\quad x\in \R^n.
\end{cases}
\end{equation}
Regarding one of the most typical equations of \eqref{Linear_Initial.Eq} when $\sigma=1$, the so-called damped wave equation, one recognizes that $L^{2} - L^{2}$ and $(L^{m}\cap L^{2})- L^{2}$ estimates for solutions with additional $L^{m}$ regularity, $m\in [1,2)$, for the initial data, were established in many papers \cite{Matsumura76, Matsumura77, DabbiccoReissig2014, Shibata2000} by dividing the phase space into small frequencies and large frequencies separately. Afterwards, the authors in \cite{NarazakiReissig2013} made a generalization to catch $L^{r_1} - L^{r_2}$ estimates, with $1\le r_1 \le r_2\le \ity$, based on investigating $L^1$ estimates for oscillating integrals, where $\delta\in(0,1)$. The crux of their method is to use the theory of modified Bessel functions connected to Fourier multipliers with oscillations appearing for wave models. More recently, the authors in \cite{DaoReissig1, DaoReissig2, DuongKainaneReissig2015, DabbiccoEbert2017, DabbiccoEbert2021} extended these aforementioned estimates to \eqref{Linear_Initial.Eq} for any $\sigma> 1$. In particular, another approach was utilized in \cite{DabbiccoEbert2017} to obtain sharp $L^{r_1}- L^{r_2}$ estimates with $1<r_1 \le r_2<\ity$, but did not include $L^1$ estimates. Later, the authors in \cite{DaoReissig1, DaoReissig2} completed such estimates for $\delta\in (0,\sigma)$ by developing the theory of modified Bessel functions associated with Fa\`{a} di Bruno's formula in the treatment of $\sigma$-evolution models instead of wave models. Unfortunately, this strategy fails in the case of visco-elastic type damping $\delta=\sigma$. Hence, the employment of the Mikhlin-H\"{o}rmander multiplier theorem in \cite{DaoReissig2} comes into play to fill out this lack. For the purpose of demonstrating $L^{r_1}- L^{r_2}$ long-time estimates with $1\le r_1 \le r_2\le \ity$, another optimal method was discussed in \cite{DabbiccoEbert2021} to separately control the oscillatory component and the diffusive component of solutions in a particular time-dependent zone of the Fourier space. \medskip

Among other things, the point worth noticing is that the asymptotic profiles of solutions to \eqref{Linear_Initial.Eq} when $\sigma=1$ were described in \cite{IkehataTakeda2019}, which have been also generalized for any $\sigma>1$ in the very recent manuscript \cite{DaoNga2024}. From these cited papers, one can understand that the shape of solutions to \eqref{Linear_Initial.Eq} strongly depends on the range of $\delta$, that is, $\delta\in [0,\sigma/2)$, $\delta=\sigma/2$ and $\delta\in (\sigma/2,\sigma]$. Actually, the authors in \cite{DabbiccoReissig2014, DaoReissig2} mentioned this observation before. They have underlined that the asymptotic properties of solutions change completely not only from the case $\delta\in [0,\sigma/2)$ to $\delta\in (\sigma/2,\sigma]$, but also according to the choice of $\mu$ in the specific case $\delta=\sigma/2$. In other words, we would say that $\delta=\sigma/2$ plays a role as a threshold to classify two different kinds of models in \eqref{Linear_Initial.Eq} and $\mu(-\Delta)^{\sigma/2} u_t$ becomes the critical damping term. More precisely, ``parabolic like models" for the case $\delta\in [0,\sigma/2)$ and ``$\sigma$-evolution like models" for the case $\delta\in (\sigma/2,\sigma]$, the so-called ``hyperbolic like models" or ``wave like models" when $\sigma=1$, have been proposed in \cite{DabbiccoReissig2014,DaoReissig2}. Roughly speaking, the asymptotic profile of the solutions to ``parabolic like models", as $t \to \ity$, is similar to that of the following anomalous diffusion equation:
\begin{equation*}
\begin{cases}
    v_t+ \nu (-\Delta)^{\sigma- \delta}v= 0, &\quad x\in \R^n,\, t \ge 0, \\
    v(0,x)= v_0(x), &\quad x\in \R^n.    
\end{cases}
\end{equation*}
for a suitable choice of data $v_0=v_0(u_0,u_1,\mu,\sigma,\delta,n)$ and a constant $\nu=\nu(\mu,\sigma,\delta,n)$. The last equation turns to the heat equation as long as we have $\sigma-\delta=1$. However, this phenomenon is no longer true for ``$\sigma$-evolution like models", i.e. some kind of wave structure and oscillations in time appear in terms of studying the asymptotic profile of the solutions.\medskip

Concerning the semi-linear equation \eqref{Initial.Eq} when $\sigma=1$, let us remember the series of previous works \cite{DabbiccoReissig2014, Nishihara2003, Narazaki2004, TodorovaYordanov2001, Zhang2001, IkehataTanizawa2005,  IkedaOgawa2016, FujiwaraIkedaWakasugi2021} in which a lot of interesting results have been done including global (in time) existence of small data solutions, blow-up phenomena of solutions in finite time and some estimates for the maximal existence time of solutions, the so-called lifespan estimates. From these achieved ones, we found out the sharp quantity of $p$ to conclude a threshold between global (in time) solutions with small data and blow-up solutions. It is called the \textit{critical exponent}. After that, the authors in \cite{DaoReissig1, DaoReissig2, DuongKainaneReissig2015, DabbiccoEbert2017, Duong2019,  DaoReissig2021, IkedaWakasugi2015} have generalized such results to \eqref{Initial.Eq} for any $\sigma>1$. Specifically, the global (in time) existence of small data solutions has been proven by using decay estimates for solutions from the corresponding linear equation \eqref{Linear_Initial.Eq} linked to several effective tools from Harmonic Analysis not only in $L^2$ norms but also in the more general $L^q$ scale with $q\in (1,\ity)$. In addition, a modified test function method for $\sigma>1$ has been introduced in place of the standard test function method for $\sigma=1$ to deal with the fractional Laplacian, a well-known nonlocal operator. By this way, one may arrive at some remarkable extensions for both blow-up results and lifespan estimates simultaneously.

\subsection{Main purpose of this paper}
In this paper, let us consider the following Cauchy problem for semi-linear $\sigma$-evolution equations with double structural damping:
\begin{equation} \label{Main.Eq}
\begin{cases}
u_{tt}+ (-\Delta)^\sigma u+ \mu_1(-\Delta)^{\sigma_1} u_t+ \mu_2(-\Delta)^{\sigma_2} u_t= |\partial_t^j u|^p, &\quad x\in \R^n,\, t \ge 0, \\
u(0,x)= u_0(x),\quad u_t(0,x)= u_1(x), &\quad x\in \R^n,
\end{cases}
\end{equation}
where $\sigma\ge 1$ is assumed to be any fractional number and $0\le \sigma_1< \sigma/2< \sigma_2\le \sigma$. The corresponding linear equation to \eqref{Main.Eq} we have in mind is
\begin{equation} \label{Linear_Main.Eq}
\begin{cases}
u_{tt}+ (-\Delta)^\sigma u+ \mu_1(-\Delta)^{\sigma_1} u_t+ \mu_2(-\Delta)^{\sigma_2} u_t= 0, &\quad x\in \R^n,\, t \ge 0, \\
u(0,x)= u_0(x),\quad u_t(0,x)= u_1(x), &\quad x\in \R^n.
\end{cases}
\end{equation}
Coming back to the very particular cases $\sigma=1$, $\sigma_1=0$ and $\sigma_2=1$ in \eqref{Main.Eq} and \eqref{Linear_Main.Eq}, the so-called wave equations with frictional damping and visco-elastic damping, the authors in \cite{IkehataTakeda2017, IkehataSawada2016} have concluded that the influence of the frictional damping is really more dominant than that of the visco-elastic one by means of studying the asymptotic profile of solutions to \eqref{Linear_Main.Eq} in the $L^2$ norm and with respect to the critical exponent for \eqref{Main.Eq} as well. Later, the author in \cite{DAbbicco2017} developed $L^1-L^1$ estimates to establish the global (in time) existence of energy solutions for any space dimensions. Several further results for the global (in time) existence of Sobolev solutions to \eqref{Main.Eq} with suitable regularity were shown in \cite{MezadekMezadekReissig2020} with the help of new inequalities from Harmonic Analysis. Speaking more about the asymptotic profile of solutions to \eqref{Linear_Main.Eq}, higher order asymptotic expansions of solutions were also indicated in \cite{IkehataMichihisa2019} thanks to the application of Taylor expansion method. Among other things, the authors in the very recent paper \cite{ChenDAbbiccoGirardi2022} have considered \eqref{Linear_Main.Eq} with $\sigma=1$, $\sigma_1 \in [0,1/2)$ and $\sigma_2 \in (\sigma/2,1]$ to derive some $L^{r_1}- L^{r_2}$ estimates with $1<r_1 \le r_2<\ity$ by expanding the kernels of solutions and employing a combination of the Mikhlin-H\"{o}rmander multiplier theorem and Hardy–Littlewood theorem for the Riesz potential. One realizes that their goal is to use the gained estimates to prove the global (in time) existence of small data solutions to \eqref{Main.Eq} with $\sigma=1$, $\sigma_1 \in [0,1/2)$ and $\sigma_2 \in (\sigma/2,1]$. Quite recently, the authors in \cite{DaoMichihisa2020, MezadekMezadekReissig2022} not only have extended these results related to the cases $\sigma>1$, $\sigma_1=0$ and $\sigma_2=\sigma$ for \eqref{Main.Eq} and \eqref{Linear_Main.Eq}, but also have analyzed the large-time behavior of global solutions by developing several techniques in the cited previous papers to work well for any fractional number. \medskip

To the best of the authors' knowledge, so far it seems that there is no any paper devoted to studying the asymptotic profile of solutions to \eqref{Main.Eq} and \eqref{Linear_Main.Eq} for any fractional number $\sigma> 1$, $\sigma_1 \in [0,\sigma/2)$ and $\sigma_2 \in (\sigma/2,\sigma]$. Hence, our main goal in this paper is, on the one hand, to investigate the asymptotic profile of solutions and their higher-order derivatives to \eqref{Linear_Main.Eq} and as a result, optimal decay rates of solutions can be obtained in the $L^2$ setting. On the other hand, we would like to indicate the global (in time) existence of small data solutions to \eqref{Main.Eq}, and moreover, sharply describe the large-time behaviors of globally achieved solutions. Throughout this paper, we can comprehensively understand the new interaction of ``parabolic like models" and ``$\sigma$-evolution like models" appearing together in \eqref{Main.Eq} and \eqref{Linear_Main.Eq}. Namely, the damping term representing ``parabolic-like models" affects more dominantly than the one of ``$\sigma$-evolution-like models" in terms of exploring the asymptotic profile in all situations of $\sigma> 1$ and for any $\sigma_1 \in [0,\sigma/2)$, $\sigma_2 \in (\sigma/2,\sigma]$. Meanwhile, the damping term representing ``$\sigma$-evolution like models" completely decides the required regularity for both the initial data and the solutions to get these asymptotic behaviors. We can say that the latter property comes from a smoothing effect appearing for some derivatives in time of solutions, which brings some advantages in the treatment of \eqref{Main.Eq}. \medskip

\noindent\textbf{Notations:}
\begin{itemize}[leftmargin=*]
\item We write $f\lesssim g$ when there exists a constant $C>0$ such that $f\le Cg$, and $f \sim g$ when $g\lesssim f\lesssim g$.
\item We denote by $\widehat{w}(t,\xi):= \mathfrak{F}_{x\rightarrow \xi}\big(w(t,x)\big)$, the Fourier transform with respect to the spatial variable of a function $w(t,x)$. Moreover, $\mathfrak{F}^{-1}$ represents the inverse Fourier transform.
\item As usual, $H^{a}$ and $\dot{H}^{a}$, with $a \ge 0$, denote Bessel and Riesz potential spaces based on $L^2$ spaces. Here $\big<D\big>^{a}$ and $|D|^{a}$ stand for the pseudo-differential operators with symbols $\big<\xi\big>^{a}$ and $|\xi|^{a}$, respectively.
\item For any $\gamma \in \R$, we denote by $[\gamma]^+:= \max\{\gamma,0\}$, its positive part. Moreover, we fix the constant $m_0:=\frac{2m}{2-m}$, that is, $\frac{1}{m_0}=\frac{1}{m}- \frac{1}{2}$ with $m \in [1,2)$.
\item For later convenience, we denote the two quantities
$$P_0:= \int_{\R^n}u_0(x)dx, \quad 
P_1:= \int_{\R^n}u_1(x)dx,$$
and the following kernel functions:
\begin{equation*}
    \mathcal{G}_0(t,x) = \mathfrak{F}^{-1}\left(e^{-|\xi|^{2(\sigma - \sigma_1)}t}\right)(t,x), \quad  
    \mathcal{G}_1(t,x) = \mathfrak{F}^{-1}\left(\dfrac{e^{-|\xi|^{2(\sigma - \sigma_1)}t}}{|\xi|^{2\sigma_1}}\right)(t,x).
\end{equation*}
\item Finally, we introduce the space
$\mathcal{D}:= \big(H^{\gamma_0}\cap L^m\big) \times \big(H^{\gamma_1}\cap L^m\big)$ with the norm
$$\|(\phi_0,\phi_1)\|_{\mathcal{D}}:=\|\phi_0\|_{H^{\gamma_0}}+ \|\phi_0\|_{L^m}+ \|\phi_1\|_{H^{\gamma_1}}+ \|\phi_1\|_{L^m} \quad \text{ with }\gamma_0, \gamma_1 \ge 0 \text{ and }m\in [1,2). $$
\end{itemize}
\medskip

\noindent\textbf{The structure of this paper is organized as follows:} We devote Section \ref{Linear estimates} to study the Cauchy linear problem \eqref{Linear_Main.Eq}. After stating the main results at the beginning of Section \ref{Linear estimates}, we present the proofs of decay estimates in Section \ref{DecayEstimates.Sec}, which play essential roles in Section \ref{GlobalExistence.Sec}, and asymptotic profiles in Section \ref{LinearAsymptotic.Sec}. Next, to begin with Section \ref{Semi-linear.Sec}, we state the main results for the Cauchy semi-linear problem \eqref{Main.Eq}, then we give the proofs of the existence of global (in time) solutions and large-time behaviors of the global solutions in Sections \ref{GlobalExistence.Sec} and \ref{Semi-LinearAsymptotic.Sec}, respectively. Finally, further remarks and open problems will be mentioned in Section \ref{Concluding.Sec}.

%==============================================================================
\section{Treatment of the corresponding linear problem} \label{Linear estimates}
In this section, let us consider the linear Cauchy problem \eqref{Linear_Main.Eq}. Our main results are as follows.

\begin{theorem}[\textbf{Decay estimates}] \label{Linear_Decay}
Assume that the initial data belongs to the space
$$ (u_0,u_1) \in \mathcal{D}_{m,s,j}^{\rm lin}:= (H^{s+2j\sigma_2}\cap L^m) \times (H^{[s+2(j-1)\sigma_2]^+}\cap L^m) $$
with $m \in [1,2)$, $j=0,1$ and $s \ge 0$. Then, the Sobolev solutions to \eqref{Linear_Main.Eq} satisfy the $(L^m \cap L^2) - L^2$ estimates
  \begin{align*}
  \|u(t,\cdot)\|_{L^2} &\lesssim 
  \left\{\begin{aligned} 
      &(1+t)^{-\frac{n}{2(\sigma-\sigma_1)}(\frac{1}{m} -\frac{1}{2}) +\frac{\sigma_1}{\sigma-\sigma_1}} \|(u_0,u_1)\|_{\mathcal{D}_{m,0,0}^{\rm lin}} & \text{if } n > 2m_0\sigma_1, \\
       & \log(e+t) \|(u_0,u_1)\|_{(L^2 \cap L^m)\times (L^2 \cap L^m)} & \text{if } n = 2m_0\sigma_1,
    \end{aligned}\right. \\
  \||D|^s u(t,\cdot)\|_{L^2} &\lesssim
 (1+t)^{-\frac{n}{2(\sigma-\sigma_1)}(\frac{1}{m} -\frac{1}{2}) -\frac{s -2\sigma_1}{2(\sigma-\sigma_1)}} \|(u_0,u_1)\|_{\mathcal{D}_{m,s,0}^{\rm lin}}\quad \text{ for }n>m_0(2\sigma_1-s), \\
\||D|^s u_t(t,\cdot)\|_{L^2} &\lesssim 
   (1+t)^{-\frac{n}{2(\sigma-\sigma_1)}(\frac{1}{m} -\frac{1}{2}) -1-\frac{s-2\sigma_1}{2(\sigma-\sigma_1)}} \|(u_0,u_1)\|_{\mathcal{D}_{m,s,1}^{\rm lin}}\quad \text{for } n \geq 1.
  \end{align*}
  and the $L^2-L^2$ estimates
  \begin{align*}
    \||D|^s u(t,\cdot)\|_{L^2} &\lesssim 
  \begin{cases}
      (1+t)^{1-\frac{s}{2(\sigma-\sigma_1)}} \|(u_0,u_1)\|_{H^{s}\times L^2} &\text{if } s < 2\sigma_1, \\
       (1+t)^{-\frac{s}{2(\sigma-\sigma_1)}+\frac{\sigma_1}{\sigma-\sigma_1}} \|(u_0,u_1)\|_{H^s\times H^{[s-2\sigma_2]^{+}}} &\text{if } s \geq 2\sigma_1,
\end{cases} \\
 \||D|^s u_t(t,\cdot)\|_{L^2} &\lesssim (1+t)^{-\frac{s}{2(\sigma-\sigma_1)}} \|(u_0,u_1)\|_{H^{s+2(\sigma-\sigma_2)} \times H^s},
      \end{align*}
for all $n \geq 1$.
\end{theorem}

\begin{theorem}[\textbf{Asymptotic profiles}]\label{Linear_Asym}
Let $j=0,1$ and $s \ge 0$. Assume that the condition $n>4\sigma_1$ and the initial data
$$ (u_0,u_1) \in \mathcal{D}_{1,s,j}^{\rm lin}:= (H^{s+2j\sigma_2}\cap L^1) \times (H^{[s+2(j-1)\sigma_2]^+}\cap L^1). $$
Then, the Sobolev solutions to \eqref{Linear_Main.Eq} satisfy the following estimates for large $t \ge 1$:
\begin{equation}
\Big\|\partial_t^j |D|^s \Big(u(t,\cdot)-P_0\,\mathcal{G}_0(t,\cdot)- P_1\,\mathcal{G}_1(t,\cdot)\Big)\Big\|_{L^2}= o\Big(t^{-\frac{n}{4(\sigma-\sigma_1)}-\frac{s}{2(\sigma-\sigma_1)}-j+ \frac{\sigma_1}{\sigma-\sigma_1}}\Big). \label{theorem1.3.1}
\end{equation}
Moreover, if $P_0 \neq 0$ or $P_1 \neq 0$, then the following estimates hold for large $t \ge 1$:
\begin{equation}
C_m\,t^{-\frac{n}{4(\sigma-\sigma_1)}-\frac{s}{2(\sigma-\sigma_1)}-j+ \frac{\sigma_1}{\sigma-\sigma_1}}\le \big\|\partial_t^j |D|^s  u(t,\cdot)\big\|_{L^2}\le C_M\,t^{-\frac{n}{4(\sigma-\sigma_1)}-\frac{s}{2(\sigma-\sigma_1)}-j+ \frac{\sigma_1}{\sigma-\sigma_1}}, \label{theorem1.3.2}
\end{equation}
where $C_m$ and $C_M$ are some suitable positive constants.
\end{theorem}

%=====================================================================
\subsection{Proof of decay estimates} \label{DecayEstimates.Sec}

Applying the Fourier transform to \eqref{Linear_Main.Eq}, we obtain
\begin{equation}\label{Linear.Eq-Fourier}
\begin{cases}
    \widehat{u}_{tt}(t,\xi) + (\mu_1|\xi|^{2\sigma_1}+\mu_2|\xi|^{2\sigma_2})\widehat{u}_t(t,\xi) + |\xi|^{2\sigma} \widehat{u}(t,\xi) = 0, &\quad \xi\in \R^n,\, t \ge 0, \\
    \widehat{u}(0,\xi) = \widehat{u}_0(\xi),\quad \widehat{u}_t(0,\xi)= \widehat{u}_1(\xi), &\quad \xi\in \R^n.
\end{cases}
\end{equation}
The characteristic equation of \eqref{Linear.Eq-Fourier} has two roots $\lambda_{1,2} =\lambda_{1,2}(\xi)$ which are given by
\begin{equation}\label{2.8}
\lambda_{1,2}(\xi) =
\begin{cases}
    \frac{1}{2}\left(-\mu_1|\xi|^{2\sigma_1} - \mu_2|\xi|^{2\sigma_2} \pm \sqrt{(\mu_1|\xi|^{2\sigma_1} + \mu_2|\xi|^{2\sigma_2})^{2} - 4 |\xi|^{2\sigma}}\right) & \text { if } \xi \in \Omega, \\
     \frac{1}{2}\left(-\mu_1|\xi|^{2\sigma_1} - \mu_2|\xi|^{2\sigma_2} \pm i \sqrt{4 |\xi|^{2\sigma}-(\mu_1|\xi|^{2\sigma_1} + \mu_2|\xi|^{2\sigma_2})^{2} }\right) & \text { if } \xi \in \mathbb{R}^n \setminus \Omega,  
\end{cases}
\end{equation}
where $\Omega = \left\{\xi \in \mathbb{R}^n: \mu_1|\xi|^{2\sigma_1} + \mu_2|\xi|^{2\sigma_2} > 2 |\xi|^{\sigma} \right\}$. Because of $0 \le \sigma_1 <\sigma/2 < \sigma_2 \le \sigma$, there exists a sufficiently small constant $\varepsilon^*>0$ such that
\begin{equation*}
    (-\ity, \varepsilon^*)\cup \left(\frac{1}{\varepsilon^*},\ity\right) \subset \Omega.
\end{equation*}
Then, taking account of the cases of small and large frequencies separately, one sees
\begin{align}
     &\lambda_{1} \sim -|\xi|^{2(\sigma-\sigma_1)},\quad 
    \lambda_{2} \sim -|\xi|^{2\sigma_1}, \quad \lambda_1-\lambda_2 \sim |\xi|^{2\sigma_1}\quad \text{ for } |\xi| \leq \varepsilon^*, \label{2.15} \\
    &\lambda_{1} \sim -|\xi|^{2(\sigma-\sigma_2)} ,\quad \lambda_{2} \sim -|\xi|^{2\sigma_2}, \quad \lambda_1-\lambda_2 \sim |\xi|^{2\sigma_2} \quad \text{ for }|\xi| \geq \frac{1}{\varepsilon^*}. \label{2.16}
\end{align} 
Let $\chi_k= \chi_k(r)$ with $k\in\{\rm L,M,H\}$ be smooth cut-off functions having the following properties:
\begin{align*}
&\chi_{\rm L}(r)=
\begin{cases}
1 &\quad \text{ if }r\le \varepsilon^*/2, \\
0 &\quad \text{ if }r\ge \varepsilon^*,
\end{cases}
\qquad
\chi_{\rm H}(r)=
\begin{cases}
1 &\quad \text{ if }r\ge 2/\varepsilon^*, \\
0 &\quad \text{ if }r\le 1/\varepsilon^*,
\end{cases} \\ 
&\text{and } \chi_{\rm M}(r)= 1- \chi_{\rm L}(r)- \chi_{\rm H}(r).
\end{align*}
We note that $\chi_{\rm M}(r)= 1$ if $\varepsilon^* \le r\le 1/\varepsilon^*$ and $\chi_{\rm M}(r)= 0$ if $r \le \varepsilon^*/2$ or $r \ge 2/\varepsilon^*$. The solution to \eqref{Linear.Eq-Fourier} is
\begin{equation}\label{2.4}
    \widehat{u}(t,\xi) = \widehat{K_0}(t,\xi) \widehat{u_0}(\xi) + \widehat{K_1}(t,\xi) \widehat{u_1}(\xi),
\end{equation}
where 
\begin{equation}\label{2.6}
\widehat{K_0}(t,\xi) = \frac{\lambda_1 e^{\lambda_2t}-\lambda_2e^{\lambda_1t}}{\lambda_1 - \lambda_2}\quad\text{ and }\quad
\widehat{K_1}(t,\xi) = \frac{e^{\lambda_1t}-e^{\lambda_2t}}{\lambda_1 - \lambda_2},
\end{equation}
so that we can write the solution to \eqref{Linear_Main.Eq} as
\begin{equation}\label{2.5}
u(t,x) = K_0(t,x) * u_0(x) + K_1(t,x) * u_1(x).
\end{equation}
We now decompose the solution \eqref{Linear_Main.Eq} into three parts localized separately to low, middle, and high frequencies, that is,
\begin{equation*}
u(t,x)= u_{\rm L}(t,x)+ u_{\rm M}(t,x)+ u_{\rm H}(t,x),
\end{equation*}
where
\begin{equation*}
    u_{k}(t,x)= \mathfrak{F}^{-1}\big(\chi_k(|\xi|)\widehat{u}(t,\xi)\big)\quad \text{ with } k\in \{\rm L,M,H\}.
\end{equation*}

\begin{proof}[\textbf{Proof of Theorem \ref{Linear_Decay}}]
In order to estimate the $L^2$ norms of solutions and their derivatives by $L^m$ norms of the data, by using Young’s inequality, we need only to estimate the $L^2$ norm of general terms of the form $|\xi|^{s}\partial_t^j\widehat{K_0}(t,\xi)$ and $|\xi|^{s}\partial_t^j\widehat{K_1}(t,\xi)$. Due to Parseval's formula, \eqref{2.4} and \eqref{2.5}, we have
\begin{equation}\label{2.18}
    \|u_{\rm L}(t,\cdot)\|_{L^{2}} = \|\chi_{\rm L}(\xi)u(t,\xi)\|_{L^{2}} \leq \big\|\chi_{\rm L}(|\xi|)\widehat{K_0}(t,\xi) \widehat{u_0}(t,\xi)\big\|_{L^2} + \big\|\chi_{\rm L}(|\xi|)\widehat{K_1}(t,\xi) \widehat{u_1}(t,\xi)\big\|_{L^2}.
\end{equation}
Applying H\"{0}lder's inequality we gain
\begin{align*}
    \big\|\chi_{\rm L}(\xi)\widehat{K_0}(t,\xi) \widehat{u_0}(t,\xi)\big\|^2_{L^2} &= \int_{\mathbb{R}^{n}}\left(\chi_{\rm L}(|\xi|)\widehat{K_0}(t,\xi)\right)^2(\widehat{u_0}(t,\xi))^2 d\xi \\
     &\leq \left(\int_{\mathbb{R}^{n}}\left(\chi_{\rm L}(|\xi|)\widehat{K_0}(t,\xi)\right)^{m_0} d\xi\right)^{\frac{2}{m_0}} \left(\int_{\mathbb{R}^{n}} \left(\widehat{u}_0(\xi)\right)^{m'} d\xi \right)^{\frac{2}{m'}}\\
     &= \big\|\chi_{\rm L}(|\xi|)\widehat{K_0}(t,\xi)\big\|_{L^{m_0}}^2 \|\widehat{u}_0\|_{L^{m'}}^2.
\end{align*}
Thanks to the Hausdorff-Young's inequality, it holds $\|\widehat{u}_0\|_{L^{m'}} \leq \|u\|_{L^m}$ with $\frac{1}{m'} + \frac{1}{m} = 1$  and $m \in [1,2)$. As a result, we obtain the estimate
\begin{equation}\label{2.19}
     \big\|\chi_{\rm L}(|\xi|)\widehat{K_0}(t,\xi) \widehat{u_0}(t,\xi)\big\|_{L^2} \leq \big\|\chi_{\rm L}(|\xi|)\widehat{K_0}(t,\xi)\big\|_{L^{m_0}} \|u_0\|_{L^m}.
\end{equation}
Similarly, we also derive
\begin{equation}\label{2.20}
    \big\|\chi_{\rm L}(|\xi|)\widehat{K_1}(t,\xi) \widehat{u_1}(t,\xi)\big\|_{L^2} \leq \big\|\chi_{\rm L}(|\xi|)\widehat{K_1}(t,\xi)\big\|_{L^{m_0}} \|u_1\|_{L^m},
\end{equation}
We introduce the following notations:
\begin{align*}
    P_1^{m_0}(j,s) &\coloneqq \int_{\mathbb{R}^{n}}|\partial_t^j \widehat{K_1}(t, \xi)|^{m_0} |\xi|^{m_0s} \chi_{\rm L}(|\xi|)d \xi, \\
    P_0^{m_0}(j,s) &\coloneqq \int_{\mathbb{R}^{n}} |\partial_t^j \widehat{K_0}(t, \xi)|^{m_0}|\xi|^{m_0s} \chi_{\rm L}(|\xi|)d \xi.
 \end{align*}
By \eqref{2.15}, we arrive at
\begin{align*}
   \left|\widehat{K_1}(t,\xi)\right|^{m_0}\chi_{\rm L}(|\xi|) &\lesssim |\xi|^{-2m_0\sigma_1}\big(e^{-m_0|\xi|^{2(\sigma-\sigma_1)}t}+e^{-m_0|\xi|^{2\sigma_1}t}\big), \\
    \left|\partial_t\widehat{K_1}(t,\xi)\right|^{m_0}\chi_{\rm L}(|\xi|)&\lesssim |\xi|^{2m_0(\sigma-2\sigma_1)} e^{-m_0|\xi|^{2(\sigma-\sigma_1)}t} + e^{-m_0|\xi|^{2\sigma_1}t}, \\
    \left|\widehat{K_0}(t,\xi)\right|^{m_0}\chi_{\rm L}(|\xi|) 
    &\lesssim |\xi|^{2m_0(\sigma-2\sigma_1)}e^{-m_0|\xi|^{2\sigma_1}t}+e^{-m_0|\xi|^{2(\sigma-\sigma_1)}t}, \\
    \left|\partial_t\widehat{K_0}(t,\xi)\right|^{m_0} \chi_{\rm L}(|\xi|) &\lesssim |\xi|^{2m_0(\sigma-\sigma_1)}\big(e^{-m_0|\xi|^{2(\sigma-\sigma_1)}t}+e^{-m_0|\xi|^{2\sigma_1}t}\big).
\end{align*}
At first, $P_1^{m_0}(0,s)$ can be controlled by
\begin{equation*}
    P_1^{m_0}(0,s) \lesssim \int_{|\xi| \leq \varepsilon^*} e^{-m_0|\xi|^{2(\sigma-\sigma_1)}t}|\xi|^{-2m_0\sigma_1+m_0s} d\xi
    \lesssim \int_{|\xi| \leq \varepsilon^*} e^{-m_0|\xi|^{2(\sigma-\sigma_1)}t}|\xi|^{-2m_0\sigma_1+m_0s +n-1} d|\xi|.
\end{equation*}
By the change of variables $\eta = \xi (t+1)^{\frac{1}{2(\sigma-\sigma_1)}}$, we obtain
\begin{align*}
    P_1^{m_0}(0,s) & \lesssim \int_{|\xi| \leq \varepsilon^*} e^{-m_0|\xi|^{2(\sigma-\sigma_1)}(t+1)}|\xi|^{-2m_0\sigma_1+m_0s +n-1} d|\xi| \\
    &\lesssim (1+t)^{-\frac{n+m_0s-2m_0\sigma_1}{2(\sigma-\sigma_1)}}\int_0^{\infty}e^{-m_0|\eta|}|\eta|^{n-1+m_0s -2m_0\sigma_1} d|\eta|\\
    &\lesssim (1+t)^{-\frac{n+m_0s-2m_0\sigma_1}{2(\sigma-\sigma_1)}} = (1+t)^{-\frac{n}{2(\sigma-\sigma_1)}-m_0\frac{s-2\sigma_1}{2(\sigma-\sigma_1)}}.
\end{align*}
The convergence of the last integral is guaranteed by the assumption $n > 2m_0\sigma_1$. In the same argument, we can deal with $P_1^{m_0}(1,s)$ as follows:
\begin{align*}
    P_1^{m_0}(1,s) \lesssim& \int_{|\xi| \leq \varepsilon^*} e^{-m_0|\xi|^{2(\sigma-\sigma_1)}(t+1)}|\xi|^{2m_0(\sigma-2\sigma_1)+n-1+m_0s}d|\xi|
    + \int_{|\xi| \leq \varepsilon^*} e^{-m_0|\xi|^{2\sigma_1}(t+1)}|\xi|^{m_0s +n-1} d|\xi|\\
    \lesssim& (1+t)^{-\frac{n+m_0s+2m_0(\sigma-2\sigma_1)}{2(\sigma-\sigma_1)}} \int_0^{\infty}e^{-m_0|\eta|}|\eta|^{2m_0(\sigma-2\sigma_1)+n-1+m_0s}d|\eta|\\
    &+ (1+t)^{-\frac{n+m_0s}{2\sigma_1}}\int_0^{\infty}e^{-m_0|\eta|}|\eta|^{m_0s+n-1}d|\eta|\\
    \lesssim& (1+t)^{-\frac{n+m_0s+2m_0(\sigma-2\sigma_1)}{2(\sigma-\sigma_1)}} = (1+t)^{-\frac{n}{2(\sigma-\sigma_1)}-m_0(\frac{s-\sigma_1}{2(\sigma-\sigma_1)}+1)}.
\end{align*}
In the next step, we used the below inequality
\begin{equation*}
    -\frac{n+m_0s+2m_0(\sigma-2\sigma_1)}{2(\sigma-\sigma_1)} \geq -\frac{n+m_0s}{2\sigma_1}.
\end{equation*}
The estimates for $P_0^{m_0}(j,s)$ can be concluded in an analogous way, namely
\begin{align*}
    P_0^{m_0}(0,s) &\lesssim \int_{|\xi| \leq \varepsilon^*}|\xi|^{2m_0(\sigma-2\sigma_1)+m_0s+n-1}e^{-m_0|\xi|^{2(\sigma-\sigma_1)}(t+1)}d|\xi| \\
    &\qquad + \int_{|\xi| \leq \varepsilon^*}|\xi|^{m_0s+n-1}e^{-m_0|\xi|^{2(\sigma-\sigma_1)}(t+1)}d|\xi|\\
    &\lesssim (1+t)^{-\frac{2m_0(\sigma-2\sigma_1)+m_0s +n}{2(\sigma-\sigma_1)}}\int_0^{\infty}e^{-m_0|\eta|}|\eta|^{2m_0(\sigma-2\sigma_1)+m_0s+n-1}d|\eta|\\
    &\qquad +(1+t)^{-\frac{n+m_0s}{2(\sigma-\sigma_1)}}\int_0^{\infty}e^{-m_0|\eta|}|\eta|^{m_0s+n-1} d|\eta|\\
    &\lesssim (1+t)^{-\frac{n+m_0s}{2(\sigma-\sigma_1)}} = (1+t)^{-\frac{n}{2(\sigma-\sigma_1)}-m_0 \frac{s}{2(\sigma-\sigma_1)}}, 
\end{align*}
and 
\begin{align*}
    P_0^{m_0}(1,s)
    &\lesssim \int_{|\xi| \leq \varepsilon^*}|\xi|^{2m_0(\sigma-\sigma_1)+m_0s+n-1}e^{-m_0|\xi|^{2(\sigma-\sigma_1)}(t+1)}d|\xi|\\
    &\lesssim (1+t)^{-\frac{n+m_0s+2m_0(\sigma-\sigma_1)}{2(\sigma-\sigma_1)}} = (1+t)^{-\frac{n}{2(\sigma-\sigma_1)}-m_0(\frac{s}{2(\sigma-\sigma_1)}+1)}.
\end{align*}
In the special case $n = 2m_0\sigma_1$, some changes are only done for estimating $P_1^{m_0}(0,0)$. Particularly, we can write 
\begin{equation*}
    \widehat{K_1}(t,\xi) = e^{-\frac{1}{2}(\mu_1|\xi|^{2\sigma_1}+\mu_2|\xi|^{2\sigma_2})t}t \frac{\sinh(At)}{At}.
\end{equation*}
with $A=A(\xi):= \frac{1}{2}\sqrt{\left(\mu_1|\xi|^{2\sigma_1}+\mu_2|\xi|^{2\sigma_2}\right)^2-4|\xi|^{2\sigma}}.$ Note that for small frequency $|\xi|$ we have $A(\xi) \sim \frac{1}{2} |\xi|^{2\sigma_1}$ and $\sinh(At) \sim At$ for $At \ll 1 $. We define the function $\rho(t)$ from the equation
\begin{equation*}
    \frac{1}{2}\sqrt{\left(\mu_1\rho^{2\sigma_1}+\mu_2\rho^{2\sigma_2}\right)^2-4\rho^{2\sigma}} = \frac{1}{t}.
\end{equation*}
Therefore, $\rho \sim t^{-\frac{1}{2\sigma_1}}$ for small $\rho$ so that one concludes 
\begin{align*}
    P_1^{m_0}(0,0) \lesssim \int_{|\xi| \leq \rho(t)} e^{-\frac{m_0}{2}(|\xi|^{2\sigma_1}+|\xi|^{2\sigma_2})t}t^{m_0} d\xi + \int_{\rho(t) \leq |\xi| \leq \varepsilon^*} |\xi|^{-2m_0\sigma_1} d\xi,
\end{align*}
because of $e^{-\frac{1}{2}(|\xi|^{2\sigma_1}+|\xi|^{2\sigma_2})t}\sinh(At) \lesssim 1$. Recalling that $n = 2m_0\sigma_1$, we may arrive at
\begin{align*}
    P_1^{m_0}(0,0) \lesssim \rho^n t^{m_0} + \log(e+ \rho^{-1}) \lesssim 1+ \log\Big(e+t^{\frac{m_0}{2\sigma_1}}\Big) \lesssim \log^{m_0}(e+t)
\end{align*}
thanks to $\rho^n \sim t^{-\frac{n}{2\sigma_1}} = t^{-m_0}$ for small $\rho$. In this way, we completed the estimates for the corresponding kernels for small frequencies.

Let us now devote to large values of $|\xi|$. By \eqref{2.16}, it follows that
\begin{align*}
    |\xi|^s|\widehat{K_0}(t,\xi)| \chi_{\rm H}(|\xi|)&\lesssim e^{-|\xi|^{2(\sigma-\sigma_2)}t}+ |\xi|^{s+2(\sigma-2\sigma_2)}e^{-|\xi|^{2\sigma_2}t},\\
    |\xi|^s|\widehat{K_1}(t,\xi)| \chi_{\rm H}(|\xi|)&\lesssim |\xi|^{s-2\sigma_2}\big(e^{-|\xi|^{2(\sigma-\sigma_2)}t}+ e^{-|\xi|^{2\sigma_2}t}\big),\\
    |\xi|^s|\partial_t\widehat{K_0}(t,\xi)| \chi_{\rm H}(|\xi|)&\lesssim |\xi|^{s+2(\sigma-\sigma_2)}\big(e^{-|\xi|^{2(\sigma-\sigma_2)}t}+ e^{-|\xi|^{2\sigma_2}t}\big),\\
    |\xi|^s|\partial_t\widehat{K_1}(t,\xi)| \chi_{\rm H}(|\xi|)&\lesssim |\xi|^{s+2(\sigma-2\sigma_2)}e^{-|\xi|^{2(\sigma-\sigma_2)}t}+e^{-|\xi|^{2\sigma_2}t},
\end{align*}
where all the terms are controlled by $e^{-ct}$. For $t \geq 1$ we obtain a better decay, meanwhile, for $ t \in (0, 1]$ we can shift $|\xi|^{s}$ to some regularity assumption for the data to estimate the last term by a constant. Thus, all the $(L^2 \cap L^m) - L^2$ estimates are proved.

Now we will prove the $L^2 - L^2$ decay estimates. We will use the boundedness of the $L^{\infty}$ for $|\xi|^{s}\partial_t^j \widehat{K_0}(t,\xi)$ and $|\xi|^{s}\partial_t^j \widehat{K_1}(t,\xi)$ with $(j,s) \in \{(0,s),(1,s)\}$, $s \geq 0$ in the zone of low frequencies. Again, by \eqref{2.15}, \eqref{2.16} and the relation
$$ |\widehat{K}_1(t,\xi)| = te^{\lambda_1 t}\int_0^1 e^{-\theta t \sqrt{\left(|\xi|^{2\sigma_1}+|\xi|^{2\sigma_2}\right)^2-4|\xi|^{2\sigma}}} d\theta $$
one finds that
\begin{align*}
    |\xi|^s|\widehat{K_0}(t,\xi)| \chi_{\rm L}(|\xi|) &\lesssim(t+1)|\xi|^{s+2(\sigma-\sigma_1)}e^{-|\xi|^{2(\sigma-\sigma_1)}t}+|\xi|^s e^{-|\xi|^{2(\sigma-\sigma_1)}t} \\
    &\lesssim (t+1)^{-\frac{s}{2(\sigma-\sigma_1)}},\\
    |\xi|^{s}|\widehat{K_1}(t,\xi)| \chi_{\rm L}(|\xi|)&\lesssim \begin{cases} 
     t|\xi|^{s}e^{-|\xi|^{2\sigma-2\sigma_1}t} & \text{ if } s < 2\sigma_1\\
    |\xi|^{s - 2\sigma_1}e^{-|\xi|^{2(\sigma-\sigma_1)}t} & \text{ if } s \geq 2\sigma_1
    \end{cases}\\
    &\lesssim 
    \begin{cases} 
      (1+t)^{1-\frac{s}{2(\sigma-\sigma_1)}} &\text{ if } s < 2\sigma_1, \\
      (1+t)^{-\frac{s}{2(\sigma-\sigma_1)}+\frac{\sigma_1}{\sigma-\sigma_1}} &\text{ if } s \geq 2\sigma_1,
\end{cases}\\
|\xi|^{s}|\partial_t\widehat{K_0}(t,\xi)|\chi_{\rm L}(|\xi|) &\lesssim (1+t)^{-1-\frac{s}{2(\sigma-\sigma_1)}},\\
|\xi|^{s}|\partial_t\widehat{K_1}(t,\xi)| \chi_{\rm L}(|\xi|)&\lesssim (1+t)^{-\frac{s}{2(\sigma-\sigma_1)}}.
\end{align*}
Hence, this completes our proof.
\end{proof}
\begin{remark}
    This remark is to point out that the middle zone $\left\{\xi \in \mathbb{R}^n: \varepsilon^*<|\xi|<1/\varepsilon^* \right\}$ does not bring any difficulty to arrive at all desired estimates. More precisely, one recognizes from \eqref{2.8} that the real part of the characteristic roots $\lambda_{1,2}$ is negative there, hence some estimates for $u_{\rm M}(t,x)$ itself and its derivatives always possess an exponential decay $e^{-ct}$ with a suitable positive constant $c$. 
\end{remark}

%=====================================================================
\subsection{Proof of asymptotic profiles}\label{LinearAsymptotic.Sec}
Before giving the proof of Theorem \ref{Linear_Asym}, let us rewrite the representation formula of solutions to (\ref{Linear.Eq-Fourier}) into several parts as follows:
$$ \widehat{u}(t,\xi)= \big(\widehat{K^1_0}(t,\xi)+ \widehat{K^2_0}(t,\xi)\big)\widehat{u_0}(\xi)+ \big(\widehat{K^1_1}(t,\xi)+ \widehat{K^2_1}(t,\xi)\big)\widehat{u_1}(\xi), $$
where
\begin{align*}
\widehat{K^1_0}(t,\xi)&= \frac{-\lambda_2 e^{\lambda_1 t}}{\lambda_1- \lambda_2}, \quad &\widehat{K^2_0}(t,\xi)&= \frac{\lambda_1 e^{\lambda_2 t}}{\lambda_1- \lambda_2}, \\ 
\widehat{K^1_1}(t,\xi)&= \frac{e^{\lambda_1 t}}{\lambda_1- \lambda_2}, \quad &\widehat{K^2_1}(t,\xi)&= \frac{-e^{\lambda_2 t}}{\lambda_1- \lambda_2}.
\end{align*}
In order to prove Theorem \ref{Linear_Asym}, the following auxiliary results come into play.

\begin{proposition} \label{lemma2.6}
Let $s\ge 0$ and $j=0,\,1$. Then, the following estimates hold for $m \in [1,2)$:
\begin{align}
\big\|\partial_t^j |D|^s \big(K^1_{0\rm L}(t,x) \ast u_0(x)\big)(t,\cdot)\big\|_{L^2}&\lesssim (1+t)^{-\frac{n}{2(\sigma-\sigma_1)}(\frac{1}{m}-\frac{1}{2})-\frac{s}{2(\sigma-\sigma_1)}-j}\|u_0\|_{L^m}, \label{lemma2.6.1} \\
\big\|\partial_t^j |D|^s \big(K^2_{0\rm L}(t,x) \ast u_0(x)\big)(t,\cdot)\big\|_{L^2}&\lesssim (1+t)^{-\frac{n}{2\sigma_1}(\frac{1}{m}-\frac{1}{2})-\frac{s}{2\sigma_1}-j-\frac{\sigma-2\sigma_1}{\sigma_1}}\|u_0\|_{L^m}, \label{lemma2.6.2}
\end{align}
for any space dimensions $n\ge 1$, and
\begin{align}
\big\|\partial_t^j |D|^s \big(K^1_{1\rm L}(t,x) \ast u_1(x)\big)(t,\cdot)\big\|_{L^2}&\lesssim (1+t)^{-\frac{n}{2(\sigma-\sigma_1)}(\frac{1}{m}-\frac{1}{2})-\frac{s}{2(\sigma-\sigma_1)}-j+\frac{\sigma_1}{\sigma-\sigma_1}}\|u_1\|_{L^m}, \label{lemma2.6.3} \\
\big\|\partial_t^j |D|^s \big(K^2_{1\rm L}(t,x) \ast u_1(x)\big)(t,\cdot)\big\|_{L^2}&\lesssim (1+t)^{-\frac{n}{2\sigma_1}(\frac{1}{m}-\frac{1}{2})-\frac{s}{2\sigma_1}-j+1}\|u_1\|_{L^m}, \label{lemma2.6.4}
\end{align}
for any space dimensions $n> 2m_0\sigma_1$. Moreover, the following estimates hold for $m \in [1,2]$:
\begin{align}
\big\|\partial_t^j |D|^s \big(K_{0\rm H}(t,x) \ast u_0(x)\big)(t,\cdot)\big\|_{L^2} &\lesssim e^{-ct}\|u_0\|_{H^{s+2j\sigma_2}}, \label{lemma2.6.5} \\
\big\|\partial_t^j |D|^s \big(K_{1\rm H}(t,x) \ast u_1(x)\big)(t,\cdot)\big\|_{L^2} &\lesssim e^{-ct}\|u_1\|_{H^{[s+2(j-1)\sigma_2]^+}}, \label{lemma2.6.6}
\end{align}
for any space dimensions $n\ge 1$, where $c$ is a suitable positive constant.
\end{proposition}

\begin{proof}
By using the asymptotic behavior of the characteristic roots \eqref{2.15} and \eqref{2.16} again, we may repeat some arguments in the proof of Theorem \ref{Linear_Decay} to conclude all the desired estimates.
\end{proof}

\begin{proposition} \label{proposition3.5}
Let $s\ge 0$ and $j=0,\,1$. Then, the following estimates hold for $m \in [1,2)$:
\begin{align}
&\Big\|\partial_t^j |D|^s \Big(\Big(K^1_{0\rm L}(t,x)- \mathfrak{F}^{-1}\Big(e^{-|\xi|^{2(\sigma-\sigma_1)}t}\chi_{\rm L}(|\xi|)\Big)\Big) \ast u_0(x)\Big)(t,\cdot)\Big\|_{L^2} \nonumber \\
&\qquad \lesssim
\begin{cases}
    (1+t)^{-\frac{n}{2(\sigma-\sigma_1)}(\frac{1}{m}-\frac{1}{2})-\frac{s}{2(\sigma-\sigma_1)}-j-\frac{\sigma-2\sigma_1}{\sigma-\sigma_1}}\|u_0\|_{L^m} &\text{ if } \sigma_1+\sigma_2>\sigma, \\
    (1+t)^{-\frac{n}{2(\sigma-\sigma_1)}(\frac{1}{m}-\frac{1}{2})-\frac{s}{2(\sigma-\sigma_1)}-j-\frac{\sigma_2-\sigma_1}{\sigma-\sigma_1}}\|u_0\|_{L^m} &\text{ if } \sigma_1+\sigma_2\le \sigma,
\end{cases} \label{proposition3.5.1}
\end{align}
for any space dimensions $n\ge 1$, and
\begin{align}
&\left\|\partial_t^j |D|^s \left(\left(K^1_{1\rm L}(t,x)- \mathfrak{F}^{-1}\left(\dfrac{e^{-|\xi|^{2(\sigma-\sigma_1)}t}}{|\xi|^{2\sigma_1}}\chi_{\rm L}(|\xi|)\right)\right) \ast u_1(x)\right)(t,\cdot)\right\|_{L^2} \nonumber \\
&\qquad \lesssim
\begin{cases}
    (1+t)^{-\frac{n}{2(\sigma-\sigma_1)}(\frac{1}{m}-\frac{1}{2})-\frac{s}{2(\sigma-\sigma_1)}-j-\frac{\sigma-3\sigma_1}{\sigma-\sigma_1}}\|u_1\|_{L^m} &\text{ if } \sigma_1+\sigma_2>\sigma, \\
    (1+t)^{-\frac{n}{2(\sigma-\sigma_1)}(\frac{1}{m}-\frac{1}{2})-\frac{s}{2(\sigma-\sigma_1)}-j-\frac{\sigma_2-2\sigma_1}{\sigma-\sigma_1}}\|u_1\|_{L^m} &\text{ if } \sigma_1+\sigma_2\le \sigma,
\end{cases} \label{proposition3.5.2}
\end{align}
for any space dimensions $n> 2m_0\sigma_1$.
\end{proposition}

\begin{proof}
In order to indicate Proposition \ref{proposition3.5}, at first, it is reasonable to prove the following  point-wise estimates:
\begin{align}
&|\xi|^s \chi_{\rm L}(|\xi|)\Bigg|\partial^j_t\Big(\widehat{K^1_0}(t,\xi)- e^{-|\xi|^{2(\sigma-\sigma_1)}t}\Big)\Bigg| \nonumber \\
&\qquad \lesssim 
\begin{cases}
    e^{-c|\xi|^{2(\sigma-\sigma_1)}t}\Big(t\,|\xi|^{s+2(2\sigma-3\sigma_1)+2j(\sigma-\sigma_1)}+ |\xi|^{s+2(\sigma-2\sigma_1)+2j(\sigma-\sigma_1)}\Big) &\text{ if }\sigma_1+\sigma_2>\sigma, \\
    e^{-c|\xi|^{2(\sigma-\sigma_1)}t}\Big(t|\xi|^{s+2(\sigma+\sigma_2-2\sigma_1)+2j(\sigma-\sigma_1)}+|\xi|^{s+2(\sigma-2\sigma_1)+2j\sigma_2}\Big) &\text{ if }\sigma_1+\sigma_2\le \sigma,
\end{cases} \label{pro3.1.1}
\end{align}
and
\begin{align}
&|\xi|^s \chi_{\rm L}(|\xi|)\left|\partial^j_t \left(\widehat{K^1_1}(t,\xi)- \frac{e^{-|\xi|^{2(\sigma-\sigma_1)}t}}{|\xi|^{2\sigma_1}}\right)\right| \nonumber \\
&\qquad \lesssim
\begin{cases}
    e^{-c|\xi|^{2(\sigma-\sigma_1)}t}\Big(t\,|\xi|^{s+4(\sigma-2\sigma_1)+2j(\sigma-\sigma_1)}+ |\xi|^{s+2(\sigma-3\sigma_1)+2j(\sigma-\sigma_1)}\Big) &\text{ if }\sigma_1+\sigma_2>\sigma, \\
    e^{-c|\xi|^{2(\sigma-\sigma_1)}t}\Big(t|\xi|^{s+2(\sigma+\sigma_2-3\sigma_1)+2j(\sigma-\sigma_1)}+|\xi|^{s+2(\sigma_2-2\sigma_1)+2j(\sigma-\sigma_1)}\Big)&\text{ if }\sigma_1+\sigma_2\le \sigma,
\end{cases} \label{pro3.1.2}
\end{align}
where $c$ is a suitable positive constant. Then, following the proof of Theorem \ref{Linear_Decay} we may conclude all the desired statements. In the first step, let us consider (\ref{pro3.1.1}) in the case $j=0$ to present $\widehat{K^1_0}(t,\xi)$ as follows: 
$$ \widehat{K^1_0}(t,\xi)= \frac{-\lambda_2 e^{\lambda_1 t}}{\lambda_1- \lambda_2}= e^{\lambda_1 t}- \frac{\lambda_1 e^{\lambda_1 t}}{\lambda_1- \lambda_2}. $$
By the mean value theorem, we obtain
$$ e^{\lambda_1 t}- e^{-|\xi|^{2(\sigma-\sigma_1)}t}= t\,\big(\lambda_1+ |\xi|^{2(\sigma-\sigma_1)}\big)e^{(\omega\lambda_1-(1-\omega)|\xi|^{2(\sigma-\sigma_1)})t}, $$
where $\omega \in [0,1]$. Consequently, we derive
\begin{align*}
    \Big|e^{\lambda_1 t}- e^{-|\xi|^{2(\sigma-\sigma_1)}t}\Big| &\le t\,\Big|-\lambda_1- |\xi|^{2(\sigma-\sigma_1)}\Big| e^{-\min\{-\lambda_1,\,|\xi|^{2(\sigma-\sigma_1)}\}t} \\
    &\le t\,\Big|-\lambda_1- |\xi|^{2(\sigma-\sigma_1)}\Big| e^{-c|\xi|^{2(\sigma-\sigma_1)}\}t}
\end{align*}
by \eqref{2.15}, where $c$ is a suitable positive constant. Using Newton's binomial theorem, we re-write $-\lambda_1$ for small frequencies in the form
$$ -\lambda_1= \frac{1}{2}\left(\big(|\xi|^{2\sigma_1}+ |\xi|^{2\sigma_2}\big)- \sqrt{\big(|\xi|^{2\sigma_1}+ |\xi|^{2\sigma_2}\big)^2-4|\xi|^{2\sigma}}\right)=: q_1\Big(1-\sqrt{1- q_2^2}\Big), $$
where we introduce $q_1:= \frac{1}{2}\big(|\xi|^{2\sigma_1}+ |\xi|^{2\sigma_2}\big)$ and $q_2:= \f{2|\xi|^\sigma}{|\xi|^{2\sigma_1}+ |\xi|^{2\sigma_2}}$. It is clear that $q_2< 1$ for small frequencies. For this reason, applying Newton's binomial theorem gives
\begin{align*}
-\lambda_1 &= q_1\left(1- \Big(1- \frac{1}{2}q_2^2- \frac{1}{8}q_2^4- o\big(q_2^4\big)\Big)\right)= q_1\left(\frac{1}{2}q_2^2+ \frac{1}{8}q_2^4+ o\big(q_2^4\big)\right) \\
&= \frac{|\xi|^{2\sigma}}{|\xi|^{2\sigma_1}+ |\xi|^{2\sigma_2}}+ \frac{|\xi|^{4\sigma}}{(|\xi|^{2\sigma_1}+ |\xi|^{2\sigma_2})^3}+ o\left(\frac{|\xi|^{4\sigma}}{(|\xi|^{2\sigma_1}+ |\xi|^{2\sigma_2})^3}\right).
\end{align*}
Hence, a standard calculation leads to
\begin{align*}
-\lambda_1- |\xi|^{2(\sigma-\sigma_1)} &= -\frac{|\xi|^{2\sigma+2(\sigma_2-\sigma_1)}}{|\xi|^{2\sigma_1}+ |\xi|^{2\sigma_2}}+ \frac{|\xi|^{4\sigma}}{(|\xi|^{2\sigma_1}+ |\xi|^{2\sigma_2})^3}+ o\left(\frac{|\xi|^{4\sigma}}{(|\xi|^{2\sigma_1}+ |\xi|^{2\sigma_2})^3}\right) \\ 
&= \frac{|\xi|^{4\sigma}}{(|\xi|^{2\sigma_1}+ |\xi|^{2\sigma_2})^3}\left(1- |\xi|^{2(\sigma_1+\sigma_2-\sigma)}\big(1+|\xi|^{2(\sigma_2-\sigma_1)}\big)^2\right)+ o\left(\frac{|\xi|^{4\sigma}}{(|\xi|^{2\sigma_1}+ |\xi|^{2\sigma_2})^3}\right) \\
&\quad \begin{cases}
    >0 &\text{ if }\sigma_1+\sigma_2> \sigma, \\
    <0 &\text{ if }\sigma_1+\sigma_2\le \sigma,
\end{cases}
\end{align*}
for small frequencies. This implies immediately the following relation
$$ \Big|-\lambda_1- |\xi|^{2(\sigma-\sigma_1)}\Big| \le
\begin{cases}
    |\xi|^{2(2\sigma-3\sigma_1)} &\text{ if }\sigma_1+\sigma_2>\sigma, \\
    |\xi|^{2(\sigma+\sigma_2-2\sigma_1)} &\text{ if }\sigma_1+\sigma_2\le \sigma.
\end{cases} $$
Therefore, we get
\begin{equation} \label{pro3.1.3}
\Big|e^{\lambda_1 t}- e^{-|\xi|^{2(\sigma-\sigma_1)}t}\Big| \lesssim 
\begin{cases}
    t\,|\xi|^{2(2\sigma-3\sigma_1)}e^{-|\xi|^{2(\sigma-\sigma_1)}t} &\text{ if }\sigma_1+\sigma_2>\sigma, \\
    t\,|\xi|^{2(\sigma+\sigma_2-2\sigma_1)}e^{-|\xi|^{2(\sigma-\sigma_1)}t} &\text{ if }\sigma_1+\sigma_2\le \sigma.
\end{cases}
\end{equation}
Thanks to the asymptotic behavior of characteristic roots for small frequencies, we may arrive at the following estimates:
\begin{align}
&|\xi|^s \chi_{\rm L}(|\xi|)\Big|\widehat{K^1_0}(t,\xi)- e^{-|\xi|^{2(\sigma-\sigma_1)}t}\Big| \nonumber \\
&\qquad \lesssim 
\begin{cases}
    e^{-|\xi|^{2(\sigma-\sigma_1)}t}\Big(t\,|\xi|^{2(2\sigma-3\sigma_1)}+ |\xi|^{2(\sigma-2\sigma_1)}\Big) &\text{ if }\sigma_1+\sigma_2>\sigma, \\
    e^{-|\xi|^{2(\sigma-\sigma_1)}t}\Big(t\,|\xi|^{2(\sigma+\sigma_2-2\sigma_1)}+ |\xi|^{2(\sigma-2\sigma_1)}\Big) &\text{ if }\sigma_1+\sigma_2\le \sigma.
\end{cases} \label{pro3.1.4}
\end{align}
Hence, the estimate (\ref{pro3.1.1}) is true for $j=0$. Now in order to estimate (\ref{pro3.1.2}) in the case $j=0$, we can re-write
$$ \widehat{K^1_1}(t,\xi)= \frac{e^{\lambda_1 t}}{\lambda_1- \lambda_2}= \frac{e^{\lambda_1 t}}{|\xi|^{2\sigma_1}}- \frac{1}{|\xi|^{2\sigma_1}}\left(1- \frac{|\xi|^{2\sigma_1}}{\lambda_1- \lambda_2}\right)e^{\lambda_1 t}. $$
Obviously, we have
$$ \widehat{K^1_1}(t,\xi)- \frac{e^{-|\xi|^{2(\sigma-\sigma_1)}t}}{|\xi|^{2\sigma_1}}= \frac{e^{\lambda_1 t}- e^{-|\xi|^{2(\sigma-\sigma_1)}t}}{|\xi|^{2\sigma_1}}- \frac{1}{|\xi|^{2\sigma_1}}\left(1- \frac{|\xi|^{2\sigma_1}}{\lambda_1- \lambda_2}\right)e^{\lambda_1 t}. $$
Moreover, we notice that it holds
\begin{align*}
    1- \frac{|\xi|^{2\sigma_1}}{\lambda_1- \lambda_2} &= \frac{2|\xi|^{2(\sigma_1+\sigma_2)}+|\xi|^{4\sigma_2}-4|\xi|^{2\sigma}}{\sqrt{(|\xi|^{2\sigma_1}+ |\xi|^{2\sigma_2})^2-4|\xi|^{2\sigma}}\big(\sqrt{(|\xi|^{2\sigma_1}+ |\xi|^{2\sigma_2})^2-4|\xi|^{2\sigma}}+ |\xi|^{2\sigma_1}\big)} \\
    &\sim 
    \begin{cases}
        -|\xi|^{2(\sigma- 2\sigma_1)} &\text{ if }\sigma_1+\sigma_2>\sigma, \\
        |\xi|^{2(\sigma_2- \sigma_1)} &\text{ if }\sigma_1+\sigma_2\le \sigma.
    \end{cases}
\end{align*}
Using again the estimate (\ref{pro3.1.3}) and the asymptotic behavior of characteristic roots for small frequencies, we may conclude
\begin{align}
&|\xi|^s \chi_{\rm L}(|\xi|)\left|\widehat{K^1_1}(t,\xi)- \frac{e^{-|\xi|^{2(\sigma-\sigma_1)}t}}{|\xi|^{2\sigma_1}}\right| \nonumber \\
&\qquad \lesssim
\begin{cases}
    e^{-c|\xi|^{2(\sigma-\sigma_1)}t}\Big(t\,|\xi|^{s+4(\sigma-2\sigma_1)}+ |\xi|^{s+2(\sigma-3\sigma_1)}\Big) &\text{ if }\sigma_1+\sigma_2>\sigma, \\
    e^{-c|\xi|^{2(\sigma-\sigma_1)}t}\Big(t\,|\xi|^{s+2(\sigma+\sigma_2-3\sigma_1)}+ |\xi|^{s+2(\sigma_2-2\sigma_1)}\Big) &\text{ if }\sigma_1+\sigma_2\le \sigma.
\end{cases} \label{pro3.1.5}
\end{align}
Therefore, the estimate (\ref{pro3.1.2}) is true for $j=0$. By analogous arguments, we also obtain the following estimates for $j=1$:
\begin{align}
&|\xi|^s \chi_{\rm L}(|\xi|)\Big|\partial_t\Big(\widehat{K^1_0}(t,\xi)- e^{-|\xi|^{2(\sigma-\sigma_1)}t}\Big)\Big| \nonumber \\
&\qquad \lesssim 
\begin{cases}
    e^{-|\xi|^{2(\sigma-\sigma_1)}t}\Big(t\,|\xi|^{2(3\sigma-4\sigma_1)}+ |\xi|^{2(2\sigma-3\sigma_1)}\Big) &\text{ if }\sigma_1+\sigma_2>\sigma, \\
    e^{-|\xi|^{2(\sigma-\sigma_1)}t}\Big(t\,|\xi|^{s+2(2\sigma+\sigma_2-3\sigma_1)}+ |\xi|^{2(\sigma+\sigma_2-2\sigma_1)}\Big) &\text{ if }\sigma_1+\sigma_2\le \sigma,
\end{cases} \label{pro3.1.6}
\end{align}
and
\begin{align}
&|\xi|^s \chi_{\rm L}(|\xi|)\left|\partial_t \left(\widehat{K^1_1}(t,\xi)- \frac{e^{-|\xi|^{2(\sigma-\sigma_1)}t}}{|\xi|^{2\sigma_1}}\right)\right| \nonumber \\
&\qquad \lesssim
\begin{cases}
    e^{-c|\xi|^{2(\sigma-\sigma_1)}t}\Big(t\,|\xi|^{s+2(3\sigma-5\sigma_1)}+ |\xi|^{s+4(\sigma-2\sigma_1)}\Big) &\text{ if }\sigma_1+\sigma_2>\sigma, \\
    e^{-c|\xi|^{2(\sigma-\sigma_1)}t}\Big(t\,|\xi|^{s+2(2\sigma+\sigma_2-4\sigma_1)}+ |\xi|^{s+2(\sigma+\sigma_2-3\sigma_1)}\Big) &\text{ if }\sigma_1+\sigma_2\le \sigma. 
\end{cases} \label{pro3.1.7}
\end{align}
Thus, it is obvious that all the estimates from (\ref{pro3.1.4}) to (\ref{pro3.1.7}) imply immediately (\ref{pro3.1.1}) and (\ref{pro3.1.2}). This completes our proof.
\end{proof}

\begin{proposition} \label{proposition3.6}
Let $s\ge 0$ and $j=0,\,1$. Then, the following estimates hold for $m \in [1,2)$:
\begin{align}
&\Big\|\partial_t^j |D|^s \Big(\Big(K^1_0(t,x)- \mathcal{G}_0(t,x)\Big) \ast u_0(x)\Big)(t,\cdot)\Big\|_{L^2} \nonumber \\
&\qquad \lesssim 
\begin{cases}
    (1+t)^{-\frac{n}{2(\sigma-\sigma_1)}(\frac{1}{m}-\frac{1}{2})-\frac{s}{2(\sigma-\sigma_1)}-j-\frac{\sigma-2\sigma_1}{\sigma-\sigma_1}}\|u_0\|_{L^m} &\text{ if } \sigma_1+\sigma_2>\sigma, \\
    (1+t)^{-\frac{n}{2(\sigma-\sigma_1)}(\frac{1}{m}-\frac{1}{2})-\frac{s}{2(\sigma-\sigma_1)}-j-\frac{\sigma_2-\sigma_1}{\sigma-\sigma_1}}\|u_0\|_{L^m} &\text{ if } \sigma_1+\sigma_2\le \sigma,
\end{cases} \nonumber \\
&\qquad \qquad + e^{-ct}\big(\|u_0\|_{H^{s+2j\sigma_2}}+ \|u_0\|_{H^{s+2j(\sigma-\sigma_1)}}\big) \label{proposition3.6.1}
\end{align}
for any space dimensions $n\ge 1$, and
\begin{align}
&\Big\|\partial_t^j |D|^s \Big(\Big(K^1_1(t,x)- \mathcal{G}_1(t,x)\Big) \ast u_1(x)\Big)(t,\cdot)\Big\|_{L^2} \nonumber \\
&\qquad \lesssim
\begin{cases}
    (1+t)^{-\frac{n}{2(\sigma-\sigma_1)}(\frac{1}{m}-\frac{1}{2})-\frac{s}{2(\sigma-\sigma_1)}-j-\frac{\sigma-3\sigma_1}{\sigma-\sigma_1}}\|u_1\|_{L^m} &\text{ if } \sigma_1+\sigma_2>\sigma, \\
    (1+t)^{-\frac{n}{2(\sigma-\sigma_1)}(\frac{1}{m}-\frac{1}{2})-\frac{s}{2(\sigma-\sigma_1)}-j-\frac{\sigma_2-2\sigma_1}{\sigma-\sigma_1}}\|u_1\|_{L^m} &\text{ if } \sigma_1+\sigma_2\le \sigma,
\end{cases} \nonumber \\
&\qquad \qquad+ e^{-ct}\big(\|u_1\|_{H^{[s+2(j-1)\sigma_2]^+}}+ \|u_1\|_{H^{[s+2j(\sigma-\sigma_1)-2\sigma_1]^+}}\big) \label{proposition3.6.2}
\end{align}
for any space dimensions $n> 2m_0\sigma_1$, where $c$ is a suitable positive constant.
\end{proposition}

\begin{proof}
At first, let us re-write the expression in the $L^2$ norm of (\ref{proposition3.6.1}) as follows:
\begin{align*}
&\partial_t^j |D|^s \Big(\Big(K^1_0(t,x)- \mathfrak{F}^{-1}\Big(e^{-|\xi|^{2(\sigma-\sigma_1)}t}\Big)\Big) \ast u_0(x)\Big) \\ 
&\qquad= \partial_t^j |D|^s \Big(\Big(K^1_{0\rm L}(t,x)- \mathfrak{F}^{-1}\Big(e^{-|\xi|^{2(\sigma-\sigma_1)}t}\chi_{\rm L}(|\xi|)\Big)\Big) \ast u_0(x)\Big) \\
&\qquad \quad+ \partial_t^j |D|^s \big(K^1_{0\rm H}(t,x) \ast u_0(x)\big)+ \partial_t^j |D|^s \Big(\mathfrak{F}^{-1}\Big(e^{-|\xi|^{2(\sigma-\sigma_1)}t}\chi_{\rm H}(|\xi|)\Big) \ast u_0(x)\Big).
\end{align*}
We notice that it holds
$$ |\xi|^s \chi_{\rm H}(|\xi|) \big|\partial_t^j e^{-|\xi|^{2(\sigma-\sigma_1)}t}\big|\lesssim |\xi|^{s+2j(\sigma-\sigma_1)}e^{-|\xi|^{2(\sigma-\sigma_1)}t}. $$
Following the proof of Theorem \ref{Linear_Decay}, we derive
\begin{equation}
\Big\|\partial_t^j |D|^s \Big(\mathfrak{F}^{-1}\Big(e^{-|\xi|^{2(\sigma-\sigma_1)}t}\chi_{\rm H}(|\xi|)\Big) \ast u_0(x)\Big)(t,\cdot)\Big\|_{L^2}\lesssim e^{-ct}\|u_0\|_{H^{s+2j(\sigma-\sigma_1)}}. \label{pro3.2.1}
\end{equation}
Therefore, combining (\ref{lemma2.6.5}), (\ref{proposition3.5.1}) and (\ref{pro3.2.1}) we may arrive at (\ref{proposition3.6.1}). In an analogous way to get (\ref{pro3.2.1}) we also obtain
\small
\begin{equation}
\left\|\partial_t^j |D|^s \left(\mathfrak{F}^{-1}\left(\dfrac{e^{-|\xi|^{2(\sigma-\sigma_1)}t}}{|\xi|^{2\sigma_1}}\chi_{\rm H}(|\xi|)\right) \ast u_1(x)\right)(t,\cdot)\right\|_{L^2}\lesssim e^{-ct}\|u_1\|_{H^{[s+2j(\sigma-\sigma_1)-2\sigma_1]^+}}. \label{pro3.2.2}
\end{equation}
\normalsize
Hence, combining (\ref{lemma2.6.6}), (\ref{proposition3.5.2}) and (\ref{pro3.2.2}) we may conclude (\ref{proposition3.6.2}). Summarizing, the proof of Proposition \ref{proposition3.6} is completed.
\end{proof}

\begin{proof}[\textbf{Proof of Theorem \ref{Linear_Asym}}]
In order to prove the asymptotic profile of solutions to \eqref{Linear_Main.Eq}, we can proceed as follows:
\begin{align*}
&\Big\|\partial_t^j |D|^s \Big(u(t,\cdot)-P_0\,\mathcal{G}_0(t,\cdot)- P_1\,\mathcal{G}_1(t,\cdot)\Big)\Big\|_{L^2} \\
&\quad \lesssim \Big\|\partial_t^j|D|^s \Big(\Big(K^1_0(t,x)- \mathcal{G}_0(t,x)\Big) \ast u_0(x)\Big)(t,\cdot)\Big\|_{L^2}+ \big\|\partial_t^j |D|^s \big(K^2_0(t,x) \ast u_0(x)\big)(t,\cdot)\big\|_{L^2} \\
&\qquad+ \big\||D|^s \big(K^2_1(t,x) \ast u_1(x)\big)(t,\cdot)\big\|_{L^2}+ \Big\|\partial_t^j|D|^s \Big(\Big(K^1_1(t,x)- \mathcal{G}_1(t,x)\Big) \ast u_1(x)\Big)(t,\cdot)\Big\|_{L^2} \\
&\qquad+ \Big\|\partial_t^j |D|^s \Big(\mathcal{G}_1(t,x) \ast u_1(x)- P_1\,\mathcal{G}_1(t,x)\Big)(t,\cdot)\Big\|_{L^2} \\
&\qquad +\Big\|\partial_t^j |D|^s \Big(\mathcal{G}_0(t,x) \ast u_0(x)- P_0\,\mathcal{G}_0(t,x)\Big)(t,\cdot)\Big\|_{L^2} \\
&\quad =: I_1+ I_2+ I_3+ I_4+ I_5+ I_6.
\end{align*}
Combining \eqref{lemma2.6.1} and \eqref{lemma2.6.5}, \eqref{lemma2.6.2} and \eqref{lemma2.6.5}, \eqref{lemma2.6.4} and \eqref{lemma2.6.6}, we derive
\begin{align*}
I_1 &\lesssim
\begin{cases}
    (1+t)^{-\frac{n}{4(\sigma-\sigma_1)}-\frac{s}{2(\sigma-\sigma_1)}-j-\frac{\sigma-2\sigma_1}{\sigma-\sigma_1}} \|u_0\|_{H^{s+2j\sigma_2}\cap L^1} &\text{ if } \sigma_1+\sigma_2>\sigma, \\
    (1+t)^{-\frac{n}{4(\sigma-\sigma_1)}-\frac{s}{2(\sigma-\sigma_1)}-j-\frac{\sigma_2-\sigma_1}{\sigma-\sigma_1}} \|u_0\|_{H^{s+2j\sigma_2}\cap L^1} &\text{ if } \sigma_1+\sigma_2\le \sigma,
\end{cases} \\ 
I_2 &\lesssim (1+t)^{-\frac{n}{4\sigma_1}-\frac{s}{2\sigma_1}-j- \frac{\sigma-2\sigma_1}{\sigma_1}} \|u_0\|_{H^{s+2j\sigma_2}\cap L^1}, \\
I_3 &\lesssim (1+t)^{-\frac{n}{4\sigma_1}- \frac{s}{2\sigma_1}-j+1} \|u_1\|_{H^{[s+2(j-1)\sigma_2]^+}\cap L^1},
\end{align*}
respectively. By \eqref{proposition3.6.2}, we arrive at
$$ I_4 \lesssim
\begin{cases}
    (1+t)^{-\frac{n}{4(\sigma-\sigma_1)}-\frac{s}{2(\sigma-\sigma_1)}-j-\frac{\sigma-3\sigma_1}{\sigma-\sigma_1}}\|u_1\|_{H^{[s+2(j-1)\sigma_2]^+}\cap L^1} &\text{ if } \sigma_1+\sigma_2>\sigma, \\
    (1+t)^{-\frac{n}{4(\sigma-\sigma_1)}-\frac{s}{2(\sigma-\sigma_1)}-j-\frac{\sigma_2-2\sigma_1}{\sigma-\sigma_1}}\|u_1\|_{H^{[s+2(j-1)\sigma_2]^+}\cap L^1} &\text{ if } \sigma_1+\sigma_2\le \sigma.
\end{cases}$$
To control $I_5$, we shall apply Lemma \ref{L^1.Lemma}. Indeed, it is clear that using the Parseval-Plancherel formula and the change of variables $\xi= t^{-\frac{1}{2(\sigma-\sigma_1)}}\eta$, one gives
\begin{align}
\Big\|\partial^j_t |D|^s \mathcal{G}_1(t,\cdot)\Big\|_{L^2} &= \Big\||\xi|^{s+2j(\sigma-\sigma_1)-2\sigma_1}e^{-|\xi|^{2(\sigma-\sigma_1)}t}\Big\|_{L^2} \nonumber \\
&= t^{-\frac{n}{4(\sigma-\sigma_1)}-\frac{s}{2(\sigma-\sigma_1)}-j+\frac{\sigma_1}{\sigma-\sigma_1}} \left(\int_0^\ity |\eta|^{2s-4\sigma_1+4j(\sigma-\sigma_1)}e^{-2|\eta|^{2(\sigma-\sigma_1)}}d\eta\right)^{\frac{1}{2}} \nonumber \\
&= C_1\,t^{-\frac{n}{4(\sigma-\sigma_1)}-\frac{s}{2(\sigma-\sigma_1)}-j+\frac{\sigma_1}{\sigma-\sigma_1}} \label{the1.1.1}
\end{align}
with the constant
$$C_1:= \left(\int_0^\ity |\eta|^{2s-4\sigma_1+4j(\sigma-\sigma_1)}e^{-2|\eta|^{2(\sigma-\sigma_1)}}d\eta\right)^{\frac{1}{2}}, $$
where we used the condition $n> 4\sigma_1$. Then, we employ Lemma \ref{L^1.Lemma} to derive
$$ I_5= o\Big(t^{-\frac{n}{4(\sigma-\sigma_1)}-\frac{s}{2(\sigma-\sigma_1)}-j+ \frac{\sigma_1}{\sigma-\sigma_1}}\Big)$$
as $t\to \ity$. Pay attention that it holds
$$ -\frac{n}{4\sigma_1}- \frac{s}{2\sigma_1}+1< -\frac{n}{4(\sigma-\sigma_1)}-\frac{s}{2(\sigma-\sigma_1)}+\frac{\sigma_1}{\sigma-\sigma_1} $$
as long as the condition $n> 4\sigma_1$ is provided. Finally, one also establishes the following estimate for $I_6$ in the same procedure:
$$ I_6= o\Big(t^{-\frac{n}{4(\sigma-\sigma_1)}-\frac{s}{2(\sigma-\sigma_1)}-j}\Big). $$
Therefore, from all the above estimates for $I_k$ with $k=1,2,\cdots,6$ we may conclude immediately (\ref{theorem1.3.1}). Then, from (\ref{theorem1.3.1}) and (\ref{the1.1.1}), we may arrive at the desired upper bound in the following way:
\begin{align*}
\big\|\partial^j_t |D|^s u(t,\cdot)\big\|_{L^2} &\le \Big\|\partial_t^j |D|^s \Big(u(t,\cdot)- P_0\,\mathcal{G}_0(t,\cdot)-P_1\,\mathcal{G}_1(t,\cdot)\Big)\Big\|_{L^2} \\
&\qquad + |P_0|\,\Big\|\partial^j_t |D|^s \mathcal{G}_0(t,\cdot)\Big\|_{L^2}+ |P_1|\,\Big\|\partial^j_t |D|^s \mathcal{G}_1(t,\cdot)\Big\|_{L^2} \\
&\le C_0|P_0|\,t^{-\frac{n}{4(\sigma-\sigma_1)}-\frac{s}{2(\sigma-\sigma_1)}-j}+ C_1|P_1|\,t^{-\frac{n}{4(\sigma-\sigma_1)}-\frac{s}{2(\sigma-\sigma_1)}-j+\frac{\sigma_1}{\sigma-\sigma_1}} \\
&\qquad + o\Big(t^{-\frac{n}{4(\sigma-\sigma_1)}-\frac{s}{2(\sigma-\sigma_1)}-j+\frac{\sigma_1}{\sigma-\sigma_1}}\Big) \\
&\le C_M\,t^{-\frac{n}{4(\sigma-\sigma_1)}-\frac{s}{2(\sigma-\sigma_1)}-j+\frac{\sigma_1}{\sigma-\sigma_1}}
\end{align*}
as $t\to \ity$, where $C_M$ is a suitable positive constant. Moreover, to indicate the lower bound, we may estimate in the following way:
\begin{align*}
\big\|\partial^j_t |D|^s u(t,\cdot)\big\|_{L^2}&\ge |P_0|\,\Big\|\partial^j_t |D|^s \mathcal{G}_0(t,\cdot)\Big\|_{L^2}+ |P_1|\,\Big\|\partial^j_t |D|^s \mathcal{G}_1(t,\cdot)\Big\|_{L^2} \\
&\qquad - \Big\|\partial_t^j |D|^s \Big(u(t,\cdot)- P_0\,\mathcal{G}_0(t,\cdot)- P_1\,\mathcal{G}_1(t,\cdot)\Big)\Big\|_{L^2} \\ 
&\ge C_0|P_0|\,t^{-\frac{n}{4(\sigma-\sigma_1)}-\frac{s}{2(\sigma-\sigma_1)}-j}+ C_1|P_1|\,t^{-\frac{n}{4(\sigma-\sigma_1)}-\frac{s}{2(\sigma-\sigma_1)}-j+ \frac{\sigma_1}{\sigma-\sigma_1}} \\
&\qquad - o\Big(t^{-\frac{n}{4(\sigma-\sigma_1)}-j-\frac{s}{2(\sigma-\sigma_1)}+ \frac{\sigma_1}{\sigma-\sigma_1}}\Big) \\
&\ge C_m\,t^{-\frac{n}{4(\sigma-\sigma_1)}-\frac{s}{2(\sigma-\sigma_1)}-j+ \frac{\sigma_1}{\sigma-\sigma_1}}
\end{align*}
as $t\to \ity$, where $C_m$ is a suitable positive constant. Summarizing, Theorem \ref{Linear_Asym} is proved.
\end{proof}

%==============================================================================
\section{Treatment of the semi-linear problems}\label{Semi-linear.Sec}
In this section, we will consider the inhomogeneous Cauchy problem for \eqref{Main.Eq}. Our main results read as follows.

\begin{theorem}[\textbf{Global existence with $j=0$}]\label{Global_Existence-1}
    Let us consider the initial data belong to the space
    $${\mathcal{D}_{m,0}^{\rm non}}:= (u_0,u_1)\in (H^{2\sigma_2}\cap L^m) \times (L^2\cap L^m)$$
    with $m \in [1,2)$. We assume for the dimension n the condition 
    \begin{equation}\label{2.27}
   n > 2m_0\sigma_1.
    \end{equation}
    Moreover, the exponent p satisfies the conditions
    \begin{equation}\label{2.28}
       p\in \left[\frac{2}{m}, \infty\right) \,\,\text{ if }\,\, n \leq 4\sigma_2 \quad\text{ or }\quad p \in \left[\frac{2}{m}, \frac{n}{n-4\sigma_2}\right) \,\,\text{ if }\,\, 4\sigma_2<n<\f{8\sigma_2}{2-m}
    \end{equation}
    and
    \begin{equation}\label{2.28-1}
        p > 1+\frac{2m\sigma}{n-2m\sigma_1}.
    \end{equation}
    Then, there exists a constant $\varepsilon_0 > 0$ such that for any small data $(u_0, u_1)$ satisfying the condition $\|(u_0,u_1)\|_{{\mathcal{D}_{m,0}^{\rm non}}} < \varepsilon_0$ we have a unique global (in time) small data solution
    $$u \in \mathcal{C}\big([0,\infty), H^{2\sigma_2}\big) \cap \mathcal{C}^1\big([0,\infty), L^2\big) $$
    to \eqref{Main.Eq}. The following decay estimates hold:
\begin{align}
\|u(t,\cdot)\|_{L^2} &\lesssim (1+t)^{-\frac{n}{2(\sigma-\sigma_1)}\left(\frac{1}{m}-\frac{1}{2}\right)+\frac{\sigma_1}{\sigma-\sigma_1}}\|(u_0,u_1)\|_{\mathcal{D}_{m,0}^{\rm non}}, \label{estimate3.1.1}\\
\||D|^{2\sigma_2}u(t,\cdot)\|_{L^2} &\lesssim (1+t)^{-\frac{n}{2(\sigma-\sigma_1)}\left(\frac{1}{m}-\frac{1}{2}\right)-\frac{\sigma_2-\sigma_1}{\sigma-\sigma_1}}\|(u_0,u_1)\|_{\mathcal{D}_{m,0}^{\rm non}},  \label{2.29} \\
\|u_t(t,\cdot)\|_{L^2} &\lesssim (1+t)^{-1}\|(u_0,u_1)\|_{\mathcal{D}_{m,0}^{\rm non}}.\label{2.30}
\end{align}
\end{theorem}

\begin{theorem}[\textbf{Global existence with $j=1$}]\label{Global_Existence-3}
Let $s > \max\{2\sigma_2 + n/2, 2(\sigma+\sigma_2-\sigma_1)\}$ and $m \in [1,2)$. We assume that the exponent p satisfies the condition
$$ p > \max\left\{\frac{2}{m},1 + s - 2 \sigma_2\right\} $$
and $n > 2m_0\sigma_1$. Then, there exists a constant $\varepsilon_0 > 0$ such that for any small data
$$\mathcal{D}_{m,1}^{\rm non}:= (u_0, u_1) \in \left(H^{s}\cap L^m\right) \times \left(H^{s-2\sigma_2}\cap L^m\right)$$
satisfying the assumption $\|(u_0, u_1)\|_{\mathcal{D}_{m,1}^{\rm non}} \leq \varepsilon_0$ we have a uniquely determined global (in time) small data energy solution
\begin{equation*}
    u \in \mathcal{C}\big([0, \infty), H^{s}\big) \cap \mathcal{C}^1\big([0,\infty), H^{s-2\sigma_2}\big)
\end{equation*}
to \eqref{Main.Eq}. The following estimates hold:
  \begin{align}
   \|u(t,\cdot)\|_{L^2} &\lesssim
 (1+t)^{-\frac{n}{2(\sigma-\sigma_1)}(\frac{1}{m} -\frac{1}{2}) +\frac{\sigma_1}{\sigma-\sigma_1}} \|(u_0,u_1)\|_{\mathcal{D}_{m,1}^{\rm non}}, \label{estimate3.3.1}\\
  \||D|^s u(t,\cdot)\|_{L^2} &\lesssim
 (1+t)^{-\frac{\sigma_2}{\sigma-\sigma_1}} \|(u_0,u_1)\|_{\mathcal{D}_{m,1}^{\rm non}}, \label{estimate3.3.2} \\
 \| u_t(t,\cdot)\|_{L^2} &\lesssim (1+t)^{-\frac{n}{2(\sigma-\sigma_1)}\left(\frac{1}{m}-\frac{1}{2}\right)+\frac{\sigma_1}{\sigma-\sigma_1}-1}\|(u_0,u_1)\|_{\mathcal{D}_{m,1}^{\rm non}}, \label{estimate3.3.3}\\
 \||D|^{s-2\sigma_2} u_t(t,\cdot)\|_{L^2} &\lesssim (1+t)^{-\frac{n}{2(\sigma-\sigma_1)}\left(\frac{1}{m}-\frac{1}{2}\right)+\frac{\sigma_1}{\sigma-\sigma_1}-1-\frac{s-2\sigma_2}{2(\sigma-\sigma_1)}}\|(u_0,u_1)\|_{\mathcal{D}_{m,1}^{\rm non}}. \label{estimate3.3.4}
\end{align}
\end{theorem}

\begin{theorem}[\textbf{Asymptotic profiles}]\label{Semi-linear_Asym}
    Consider that the assumptions of Theorem \ref{Global_Existence-1} hold, the global (in time) small data energy solutions to \eqref{Main.Eq} with $j = 0$ satisfy the following estimates for large $t \geq 1$:
    \begin{align}\label{eq:Semi-linear_Asym}
        \left\|\partial^j_t |D|^{2k\sigma_2} \Big(u(t, \cdot) - P_0 \mathcal{G}_0(t, \cdot) - (P_1+M_0) \mathcal{G}_1(t, \cdot)\Big)\right\|_{L^2} 
        = o\Big(t^{-\frac{n}{4(\sigma-\sigma_1)} - \frac{k\sigma_2}{\sigma-\sigma_1} - j + \frac{\sigma_1}{\sigma-\sigma_1}}\Big),
    \end{align}
    where $j, k=0,1$ and $(j,k) \ne (1,1)$.
    
    Moreover, when the assumptions of Theorem \ref{Global_Existence-3} hold, the global (in time) small data energy solutions to \eqref{Main.Eq} with $j=1$ satisfy the following estimates for large $t \geq 1$:
     \begin{align}
        \left\||D|^{ks} \Big(u(t, \cdot) - P_0 \mathcal{G}_0(t, \cdot) - (P_1 + M_1) \mathcal{G}_1(t, \cdot)\Big)\right\|_{L^2} 
        &= o\Big(t^{-\frac{n}{4(\sigma-\sigma_1)} - \frac{ks}{2(\sigma-\sigma_1)} + \frac{\sigma_1}{\sigma-\sigma_1}}\Big),\label{eq:Semi-linear_Asym2-1}\\
        \left\|\partial_t|D|^{k(s-2\sigma_2)} \Big(u(t, \cdot) - P_0 \mathcal{G}_0(t, \cdot) - (P_1 + M_1) \mathcal{G}_1(t, \cdot)\Big)\right\|_{L^2} 
        &= o\Big(t^{-\frac{n}{4(\sigma-\sigma_1)} - \frac{k(s-2\sigma_2)}{2(\sigma-\sigma_1)} - 1 + \frac{\sigma_1}{\sigma-\sigma_1}}\Big)\label{eq:Semi-linear_Asym2-2},
    \end{align}
    where $k=0,1$. In both estimates, we denote the quantities:
    \begin{equation*}
        M_0 \coloneqq \int^{\infty}_{0} \int_{\mathbb{R}^n} |u(\tau, y)|^p dy d\tau, \quad 
        M_1 \coloneqq \int^{\infty}_{0} \int_{\mathbb{R}^n} |u_t(\tau, y)|^p dy d\tau.
    \end{equation*}
\end{theorem}

%=====================================================================
\subsection{Proof of global existence results}\label{GlobalExistence.Sec}
In this section, we will use the decay estimates for solutions to \eqref{Linear_Main.Eq}, which are obtained in Theorems \ref{Linear_Decay} to prove the corresponding Theorems \ref{Global_Existence-1}. Our main tools are Duhamel’s principle and Gagliardo–Nirenberg inequality. By using fundamental solutions, we write the solution of \eqref{Linear_Main.Eq} in the form 
$$
u(t,x) = K_0(t,x) * u_0(x) + K_1(t,x) * u_1(x)
$$
then the solution to $(\ref{Main.Eq})$ becomes
\begin{align}
    u(t,x) &= K_0(t,x) \ast_x u_0(x) + K_1(t,x) \ast_x u_1(x) + \int_0^t K_1(t-\tau,x) \ast_x |\partial_t^ju(\tau,x)|^p d\tau \notag\\
    &:= u^{\rm lin}(t,x) + u^{\rm non}(t,x).\label{2.31}
\end{align}
for $j = 0,1$. We introduce the space of data  $\mathcal{D}_m^{\rm non} := ( H^{s} \cap L^m) \times (H^{[s-2\sigma_2]^{+}} \cap L^m)$ with $m \in [1,2)$ and the function spaces $X(t) := \mathcal{C}([0,t], H^{s}) \cap \mathcal{C}^1([0,t], H^{[s-2\sigma_2]^{+}}) $ for all $t > 0$ with the norm
\begin{align*}
\|u(\tau, \cdot)\|_{X(t)}&:=\sup _{0 \leq \tau \leq t}\Big(f_1(\tau)^{-1}\|u(\tau, \cdot)\|_{L^2}+f_2(\tau)^{-1}\left\||D|^{s} u(\tau, \cdot)\right\|_{L^2}\\
&\hspace{2cm}+ g_1(\tau)^{-1}\left\|u_t(\tau, \cdot)\right\|_{L^2}+g_2(\tau)^{-1}\||D|^{s-2\sigma_2}u_t\|_{H^{[s-2\sigma_2]^{+}}}\Big),
\end{align*}
We define for any $u \in X(t)$ the operator 
\begin{align*}
    \mathcal{N} \text{ : } &u \in X(t) \longrightarrow \mathcal{N}[u] \in X(t)  \\
    &\mathcal{N}[u](t,x) = u^{\rm lin}(t,x)+u^{\rm non}(t,x)
\end{align*}
with $j = 0,1$. The mapping into $X(t)$ follows from estimate 
\begin{equation}\label{2.32}
    \|\mathcal{N}[u](t,\cdot)\|_{X(t)} \lesssim \|(u_0,u_1)\|_{\mathcal{D}_m^{\rm non}} + \|u(t,\cdot)\|_{X(t)}^p
\end{equation}
Moreover, we show the Lipschitz property 
\begin{equation}\label{2.33}
    \|\mathcal{N}[u](t,\cdot)-\mathcal{N}[v](t,\cdot)\|_{X(t)} \lesssim \|u(t,\cdot)-v(t,\cdot)\|_{X(t)}\left(\|u(t,\cdot)\|_{X(t)}^{p-1} + \|v(t,\cdot)\|_{X(t)}^{p-1}\right)
\end{equation}
If we can prove \eqref{2.32} and \eqref{2.33}, then standard arguments imply, on the one hand, the local (in time) solutions for arbitrary data and, on the other hand, the global (in time) solutions for small data as well. The decay estimates for the solution and its energy follow immediately with the definition of the norm in $X(t)$.

\begin{remark}\label{remark3.1}
    Basically, if $n = 2m_0\sigma_1$ then due to Theorem \ref{Linear_Decay}, the coefficient of $\|u(\tau,\cdot)\|_{L^2}$ is $(\log(e+t))^{-1}$. Because this term brings no additional difficulties, we will ignore it.
\end{remark}

%=====================================================================
\begin{proof}[\textbf{Proof of Theorem \ref{Global_Existence-1}}] Taking $s = 2\sigma_2$ we introduce the data space $\mathcal{D}_m^{\rm non}:= \mathcal{D}_{m,0}^{\rm non}$ and solution space
\begin{equation*}
    X(t) := \mathcal{C}([0,t], H^{2\sigma_2}) \cap \mathcal{C}^1([0,t], L^2),
\end{equation*}
where we choose
\begin{align*}
    f_1(\tau) &= (1+\tau)^{-\frac{n}{2(\sigma-\sigma_1)}(\frac{1}{m}-\frac{1}{2})+\frac{\sigma_1}{\sigma-\sigma_1}},\\
    f_2(\tau) &= (1+\tau)^{-\frac{n}{2(\sigma-\sigma_1)}(\frac{1}{m}-\frac{1}{2})-\frac{\sigma_2-\sigma_1}{\sigma-\sigma_1}},\\
    g_1(\tau) &= 1+\tau\quad \text{ and }\quad g_2(\tau) = 0.
\end{align*}

\noindent\textit{Let us first prove \eqref{2.32}.} From Theorem \ref{Linear_Decay}, combined with the following estimates:
\begin{equation*}
    (1+\tau)^{-\frac{n}{2(\sigma-\sigma_1)}(\frac{1}{m}-\frac{1}{2})-1+\frac{\sigma_1}{\sigma-\sigma_1}} \lesssim (1+\tau)^{-1},
\end{equation*}
we obtain
\begin{equation*}
     \|u^{\rm lin}\|_{X(t)} \lesssim \|(u_0, u_1)\|_{\mathcal{D}_m^{\rm non}} \text{ for } s \geq 0.
\end{equation*}
For this reason, to prove the estimate (\ref{2.32}), we only need to prove the following estimate:
\begin{equation}\label{2.99}
    \|u^{\rm non}\|_{X(t)} \lesssim \|u\|_{X(t)}^p.
\end{equation}
To get started, we use two different strategies for $\tau \in [0, \frac{t}{2}]$ and $\tau \in [\frac{t}{2}, t]$ to control the integral in $u^{\rm non}$. More precisely, we use the $(L^m \cap L^2)-L^2$ estimates if $\tau \in [0,\frac{t}{2}]$ and we use $L^2-L^2$ estimates if $\tau \in [\frac{t}{2}, t]$. Taking into consideration the estimates from Theorem \ref{Linear_Decay}, we get
\begin{align}
    \|u^{\rm non}(t,\cdot)\|_{L^2} &\lesssim \int_0^t (1+t-\tau)^{-\frac{n}{2(\sigma-\sigma_1)}(\frac{1}{m}-\frac{1}{2})+\frac{\sigma_1}{\sigma-\sigma_1}}\||u(\tau,x)|^p \|_{L^m \cap L^2}d\tau,  \label{2.34} \\
     \||D|^{2\sigma_2}u^{\rm non}(t,\cdot)\|_{L^2} &\lesssim \int_0^{\frac{t}{2}} (1+t-\tau)^{-\frac{n}{2(\sigma-\sigma_1)}(\frac{1}{m} - \frac{1}{2})-\frac{\sigma_2-\sigma_1}{\sigma-\sigma_1}}\||u(\tau,x)|^p\|_{L^m \cap L^2} d\tau\\
     &\quad+ \int_{\frac{t}{2}}^t (1+t-\tau)^{-\frac{\sigma_2-\sigma_1}{\sigma-\sigma_1}}\||u(\tau,x)|^p\|_{L^2}d\tau, \label{2.35} \\
    \|u_t^{\rm non}(t,\cdot)\|_{L^2} &\lesssim \int_0^{t} (1+t-\tau)^{-\frac{n}{2(\sigma-\sigma_1)}(\frac{1}{m}-\frac{1}{2})+\frac{\sigma_1}{\sigma-\sigma_1}-1}\||u(\tau,x)|^p \|_{L^m \cap L^2}d\tau. \label{2.36}
\end{align}
It is required to estimate $|u(\tau,\cdot)|^p$ in $L^m \cap L^2$ and in $L^2$. Taking account of
\begin{align*}
    \||u(\tau,\cdot)|^p\|_{L^m \cap L^2} &\lesssim \|u(\tau,\cdot)\|_{L^{mp}}^p + \|u(\tau,\cdot)\|_{L^{2p}}^p, \\
    \||u(\tau,\cdot)|^p\|_{L^2} &\lesssim \|u(\tau, \cdot)\|_{L^{2p}}^p,
\end{align*}
reduces the requirement to estimate $\|u(\tau,\cdot)\|_{L^{2p}}$ and $\|u(\tau,\cdot)\|_{L^{mp}}$. Due to $\eqref{2.28}$ and applying the Gagliardo–Nirenberg inequality, we have
\begin{equation}\label{2.37}
    \begin{aligned}
       \|u(\tau,\cdot)\|_{L^{mp}}^p &\lesssim \|u(\tau,\cdot)\|_{L^2}^{p(1-\theta(mp))} \||D|^{2\sigma_2}u(\tau,\cdot)\|_{L^2}^{p\theta(mp)}\\
       &\lesssim \left((1+\tau)^{-\frac{\sigma_2\theta(mp)}{\sigma-\sigma_1}}\|u(\tau,\cdot)\|_{L^2}+ (1+\tau)^{\frac{\sigma_2(1-\theta(mp))}{\sigma-\sigma_1}}\||D|^{2\sigma_2}u(\tau,\cdot)\|_{L^2}\right)^p\\
       &\lesssim  (1+\tau)^{-\frac{p}{2(\sigma-\sigma_1)}(n(\frac{1}{m} -\frac{1}{2})-2\sigma_1+2\sigma_2\theta(mp))}\|u\|_{X(t)}^p\\
       &= (1+\tau)^{-\frac{n}{2m(\sigma-\sigma_1)}(p-1)+\frac{p\sigma_1}{\sigma-\sigma_1}}\|u\|_{X(t)}^p .
    \end{aligned}
\end{equation}
provided that (\ref{2.28}) is satisfied. By a similar proof, we have 
\begin{equation}\label{2.38}
    \|u(\tau,\cdot)\|_{L^{2p}}^p \lesssim \|u\|_{X(t)}^p (1+\tau)^{-\frac{pn}{2(\sigma-\sigma_1)}(\frac{1}{m}-\frac{1}{2p})+\frac{p\sigma_1}{\sigma-\sigma_1}}.
\end{equation}
In this way, we obtain 
\begin{equation}\label{2.39}
    \||u(\tau,x)|^p\|_{L^m \cap L^2} \lesssim  \|u\|_{X(t)}^p (1+\tau)^{-\frac{n}{2m(\sigma-\sigma_1)}(p-1)+\frac{p\sigma_1}{\sigma-\sigma_1}},
\end{equation}
provided that (\ref{2.28}) is satisfied, whereas
\begin{equation}\label{2.40}
    \||u(\tau,x)|^p\|_{L^2} \lesssim \|u\|_{X(t)}^p (1+\tau)^{-\frac{pn}{2(\sigma-\sigma_1)}(\frac{1}{m}-\frac{1}{2p})+\frac{p\sigma_1}{\sigma-\sigma_1}}.
\end{equation}
Thus, we have
\begin{align*}
    \|u^{\rm non}(t,\cdot)\|_{L^2} &\lesssim (1+t)^{-\frac{n}{2(\sigma-\sigma_1)}(\frac{1}{m}-\frac{1}{2})+\frac{\sigma_1}{\sigma-\sigma_1}}\|u\|_{X(t)}^p \int_0^{\frac{t}{2}} (1+\tau)^{-\frac{n}{2m(\sigma-\sigma_1)}(p-1)+\frac{p\sigma_1}{\sigma-\sigma_1}}d\tau\\
    &\quad+ (1+t)^{\frac{p\sigma_1}{\sigma-\sigma_1}-\frac{n}{2m(\sigma-\sigma_1)}(p-1)}\|u\|_{X(t)}^p 
    \int_{\frac{t}{2}}^t (1+t-\tau)^{-\frac{n}{2(\sigma-\sigma_1)}(\frac{1}{m}-\frac{1}{2})+\frac{\sigma_1}{\sigma-\sigma_1}}d\tau, \\
    \|u_t^{\rm non}(t,\cdot)\|_{L^2} &\lesssim (1+t)^{-\frac{n}{2(\sigma-\sigma_1)}(\frac{1}{m}-\frac{1}{2})+\frac{\sigma_1}{\sigma-\sigma_1} -1}\|u\|_{X(t)}^p \int_0^{\frac{t}{2}} (1+\tau)^{-\frac{n}{2m(\sigma-\sigma_1)}(p-1)+\frac{p\sigma_1}{\sigma-\sigma_1}}d\tau\\
    &\quad+ (1+t)^{\frac{p\sigma_1}{\sigma-\sigma_1}-\frac{n}{2m(\sigma-\sigma_1)}(p-1)}\|u\|_{X(t)}^p \int_{\frac{t}{2}}^t (1+t-\tau)^{-\frac{n}{2(\sigma-\sigma_1)}(\frac{1}{m}-\frac{1}{2})+\frac{\sigma_1}{\sigma-\sigma_1} -1}d\tau, \\
    \||D|^{2\sigma_2}u^{\rm non}(t,\cdot)\|_{L^2} &\lesssim (1+t)^{-\frac{n}{2(\sigma-\sigma_1)}(\frac{1}{m} - \frac{1}{2})-\frac{\sigma_2-\sigma_1}{\sigma-\sigma_1}}\|u\|_{X(t)}^p \int_0^{\frac{t}{2}} (1+\tau)^{-\frac{n}{2m(\sigma-\sigma_1)}(p-1)+\frac{p\sigma_1}{\sigma-\sigma_1}}d\tau\\
    &\quad+ (1+t)^{\frac{p\sigma_1}{\sigma-\sigma_1}-\frac{pn}{2(\sigma-\sigma_1)}(\frac{1}{m}-\frac{1}{2p})}\|u\|_{X(t)}^p\int_{\frac{t}{2}}^t (1+t-\tau)^{-\frac{\sigma_2-\sigma_1}{\sigma-\sigma_1}}d\tau,
    \end{align*}
due to $(\ref{2.28})$, we have
$$ (1+\tau)^{-\frac{n}{2m(\sigma-\sigma_1)}(p-1)+\frac{p\sigma_1}{\sigma-\sigma_1}} < (1+\tau)^{-1}. $$
Moreover, $0 < \frac{\sigma_2 -\sigma_1}{\sigma-\sigma_1} < 1$ so that we can also estimate 
\begin{equation*}
\begin{aligned}
    &(1+t)^{-\frac{pn}{2(\sigma-\sigma_1)}(\frac{1}{m}-\frac{1}{2p})+\frac{p\sigma_1}{\sigma-\sigma_1}}\int_{\frac{t}{2}}^t (1+t-\tau)^{-\frac{\sigma_2-\sigma_1}{\sigma-\sigma_1}}d\tau \\
    &\quad\lesssim (1+t)^{1-\frac{pn}{2(\sigma-\sigma_1)}(\frac{1}{m}-\frac{1}{2p})-\frac{\sigma_2-\sigma_1}{\sigma-\sigma_1}+\frac{p\sigma_1}{\sigma-\sigma_1}} \lesssim (1+t)^{-\frac{n}{2(\sigma-\sigma_1)}(\frac{1}{m} - \frac{1}{2})-\frac{\sigma_2-\sigma_1}{\sigma-\sigma_1}},
  \end{aligned}
\end{equation*}
using again $(\ref{2.28})$. On the other hand, we also have the following relations:
\begin{align*}
    (1+t)^{-\frac{n}{2m(\sigma-\sigma_1)}(p-1)+\frac{p\sigma_1}{\sigma-\sigma_1}} \int_{\frac{t}{2}}^t (1+t-\tau)^{-\frac{n}{2(\sigma-\sigma_1)}(\frac{1}{m}-\frac{1}{2})+\frac{\sigma_1}{\sigma-\sigma_1} -1}d\tau \lesssim (1+t)^{-1}.
\end{align*}
Thus, (\ref{2.99}) has been proved.\medskip

\noindent\textit{Now, we are going to prove $(\ref{2.33})$.} We remark that
\begin{equation*}
    \|\partial_t^j|D|^{s}(\mathcal{N}[u] - \mathcal{N}[v])\|_{X(t)} = \left\|\int_0^t \partial_t^j|D|^{s}G_1(t-\tau,x) *\left(|u|^p-|v|^p\right)d\tau\right\|_{X(t)}.
\end{equation*}
with $(j,s) \in \{(1,0),(0,0), (0, 2\sigma_2)\}$. Due to $(\ref{2.28})$, applying Gagliardo–Nirenberg inequality and Holder's inequality, we have
\begin{align*}
    \left\||u|^p - |v|^p\right\|_{L^m} &\lesssim \|u-v\|_{L^{mp}}\left(\|u(\tau,\cdot)\|_{L^{mp}}^{p-1}+\|v(\tau,\cdot)\|_{L^{mp}}^{p-1}\right)\\
    &\lesssim (1+\tau)^{-\frac{n}{2(\sigma-\sigma_1)}(p-1)+\frac{p\sigma_1}{\sigma-\sigma_1}} \|u-v\|_{X(t)}\left(\|u\|_{X(t)}^{p-1}+\|v\|_{X(t)}^{p-1}\right), \\
    \left\||u|^p - |v|^p\right\|_{L^2} &\lesssim \|u-v\|_{L^{2p}}\left(\|u(\tau,\cdot)\|_{L^{2p}}^{p-1}+\|v(\tau,\cdot)\|_{L^{2p}}^{p-1}\right)\\
    &\lesssim (1+\tau)^{-\frac{pn}{2(\sigma-\sigma_1)}(\frac{1}{m}-\frac{1}{2p})+\frac{p\sigma_1}{\sigma-\sigma_1}} \|u-v\|_{X(t)}\left(\|u\|_{X(t)}^{p-1}+\|v\|_{X(t)}^{p-1}\right).
\end{align*}
In this way, we may conclude the proof of $(\ref{2.33})$.
\end{proof}

%===================================================================
%...................................................................
\begin{proof}[\textbf{Proof of Theorem \ref{Global_Existence-3}}]
    We introduce the data space $\mathcal{D}_m^{\rm non}:= \mathcal{D}_{m,1}^{\rm non}$ and solution space
\begin{equation*}
    X(t) := \mathcal{C}([0,t], H^{s}) \cap \mathcal{C}^1([0,t], H^{s-2\sigma_2}),
\end{equation*}
where we choose
\begin{align*}
    f_1(\tau)&:= (1+\tau)^{-\frac{n}{2(\sigma-\sigma_1)}(\frac{1}{m}-\frac{1}{2})+\frac{\sigma_1}{\sigma-\sigma_1}},\\
    f_2(\tau) &:= (1+\tau)^{-\frac{\sigma_2}{\sigma-\sigma_1}},\\
    g_1(\tau)&:= (1+\tau)^{-\frac{n}{2(\sigma-\sigma_1)}(\frac{1}{m}-\frac{1}{2})-1+\frac{\sigma_1}{\sigma-\sigma_1}},\\
    g_2(\tau) &:= (1+\tau)^{-\frac{n}{2(\sigma-\sigma_1)}(\frac{1}{m}-\frac{1}{2})-1+\frac{\sigma_1}{\sigma-\sigma_1}-\frac{s-2\sigma_2}{2(\sigma-\sigma_1)}}.
\end{align*}

\textit{First, we prove the estimate \eqref{2.32}}. From Theorem \ref{Linear_Decay}, combined with the following estimates:
\begin{equation*}
    (1+\tau)^{-\frac{n}{2(\sigma-\sigma_1)}(\frac{1}{m}-\frac{1}{2})-\frac{s}{2(\sigma-\sigma_1)}+\frac{\sigma_1}{\sigma-\sigma_1}} \lesssim (1+\tau)^{-\frac{\sigma_2}{\sigma-\sigma_1}},
\end{equation*}
we can obtain
\begin{equation*}
     \|u^{\rm lin}\|_{X(t)} \lesssim \|(u_0, u_1)\|_{\mathcal{D}_m^{\rm non}} \text{ for } s \geq 0.
\end{equation*}
For this reason, to prove the estimate \eqref{2.32}, we only need to show the following estimate:
\begin{equation}\label{2.100}
    \|u^{\rm non}\|_{X(t)} \lesssim \|u\|_{X(t)}^p.
\end{equation}
Using the estimate $(L^m \cap L^2)-L^2$ from theorem \ref{Linear_Decay} we have
\begin{align*}
    \|u^{\rm non}(t,\cdot)\|_{L^2} &\lesssim \int_0^{t} (1+t-\tau)^{-\frac{n}{2(\sigma-\sigma_1)}(\frac{1}{m} - \frac{1}{2})+ \frac{\sigma_1}{\sigma-\sigma_1}}\| \left|u_t(\tau,x)\right|^p\|_{L^m \cap L^2} d\tau.
  \end{align*}
We get the following relations:
\begin{equation*}
    \||u_t(\tau,\cdot)|^p\|_{L^2 \cap L^m} \lesssim \|u_t(\tau,\cdot)\|_{L^{2p}}^p + \|u_t(\tau,\cdot)\|_{L^{mp}}^p \quad\text{ and }\quad \||u_t(\tau, \cdot)|^p\|_{L^2} = \|u_t(\tau,\cdot)\|_{L^{2p}}^p.
\end{equation*}
Applying the fractional Gagliardo-Nirenberg inequality implies
\begin{align*}
    \||u_t(\tau,\cdot)|^p\|_{L^2 \cap L^m} &\lesssim (1+\tau)^{-\frac{n}{2m(\sigma-\sigma_1)}(p-1)-\frac{p(\sigma-2\sigma_1)}{\sigma-\sigma_1}} \|u\|_{X(t)}^p, \\
    \||u_t(\tau,\cdot)|^p\|_{L^2} &\lesssim (1+\tau)^{-\frac{pn}{2(\sigma-\sigma_1)}(\frac{1}{m}-\frac{1}{2p})-\frac{p(\sigma-2\sigma_1)}{\sigma-\sigma_1}} \|u\|_{X(t)}^p.
\end{align*}
provided that $p \in \left[\frac{2}{m}, \infty \right)$ since $s > 2\sigma_2 + \frac{n}{2}$. From this, we have the following conclusion:
\begin{align*}
    \|u^{\rm non}(t,\cdot)\|_{L^2}
    \lesssim& \|u\|_{X(t)}^p \int_0^t (1+t-\tau)^{-\frac{n}{2(\sigma-\sigma_1)}\left(\frac{1}{m} - \frac{1}{2}\right)+ \frac{\sigma_1}{\sigma-\sigma_1}}(1+\tau)^{-\frac{n}{2m(\sigma-\sigma_1)}(p-1)-\frac{p(\sigma-2\sigma_1)}{\sigma-\sigma_1}} d\tau \\
   \lesssim& (1+t)^{-\frac{n}{2(\sigma-\sigma_1)}\left(\frac{1}{m} - \frac{1}{2}\right)+ \frac{\sigma_1}{\sigma-\sigma_1}}\|u\|_{X(t)}^p \int_0^{\frac{t}{2}} (1+\tau)^{-\frac{n}{2m(\sigma-\sigma_1)}(p-1)-\frac{p(\sigma-2\sigma_1)}{\sigma-\sigma_1}} d\tau \\
   &+ (1+t)^{-\frac{n}{2m(\sigma-\sigma_1)}(p-1)-\frac{p(\sigma-2\sigma_1)}{\sigma-\sigma_1}}\|u\|_{X(t)}^p \int_{\frac{t}{2}}^t (1+t-\tau)^{-\frac{n}{2(\sigma-\sigma_1)}\left(\frac{1}{m} - \frac{1}{2}\right)+ \frac{\sigma_1}{\sigma-\sigma_1}} d\tau.
\end{align*}
Because $p > 1+s-2\sigma_2$ and $n > 2m_0\sigma_1$ and $s > \frac{n}{2}+2\sigma_2$, we have
\begin{align}\label{equation3.3.1}
    p > 1+s-2\sigma_2 > 1+\frac{m(s-2\sigma_2)}{n+2m(\sigma-2\sigma_1)} > 1+\frac{2m\sigma_1}{n+2m(\sigma-2\sigma_1)}.
\end{align}
This leads to 
\begin{align}\label{equation3.3.2}
    -\frac{n}{2m(\sigma-\sigma_1)}(p-1)-\frac{p(\sigma-2\sigma_1)}{\sigma-\sigma_1} < -1
\end{align}
and
\begin{align*}
    -\frac{n}{2m(\sigma-\sigma_1)}(p-1)-\frac{p(\sigma-2\sigma_1)}{\sigma-\sigma_1} < -\frac{n}{2(\sigma-\sigma_1)}\left(\frac{1}{m} - \frac{1}{2}\right) + \frac{\sigma_1}{\sigma-\sigma_1}.
\end{align*}
From here, we have a conclusion
\begin{align}
    \|u^{\rm non}(t,\cdot)\|_{L^2} \lesssim (1+t)^{-\frac{n}{2(\sigma-\sigma_1)}+\frac{\sigma_1}{\sigma-\sigma_1}}\|u\|_{X(t)}^p.\label{3.3.1}
\end{align}
Next, we estimate $\||D|^su^{\rm non}(t,\cdot)\|_{L^2}$. We have
\begin{align*}
    \||D|^su^{\rm non}u(t,\cdot)\|_{L^2} &\lesssim \int_0^{\frac{t}{2}} (1+t-\tau)^{-\frac{n}{2(\sigma-\sigma_1)}(\frac{1}{m} -\frac{1}{2})-\frac{s}{2(\sigma-\sigma_1)}+\frac{\sigma_1}{\sigma-\sigma_1}} \| |u_t(\tau,\cdot)|^p \|_{L^m \cap L^2 \cap \dot{H}^{s-2\sigma_2}} d\tau\\
    &\quad+ \int_{\frac{t}{2}}^t (1+t-\tau)^{-\frac{s}{2(\sigma-\sigma_1)}+\frac{\sigma_1}{\sigma-\sigma_1}} \| |u_t(\tau,\cdot)|^p \|_{L^2 \cap \dot{H}^{s-2\sigma_2}} d\tau. 
\end{align*}
Therefore, we need to estimate the norm $\||u_t(\tau,\cdot)|^p\|_{\dot{H}^{s-2\sigma_2}}$. Using Proposition \ref{FractionalPowers} and Proposition \ref{Embedding} with $s^{*} < \frac{n}{2} < s-2\sigma_2 < p$, we obtain
\begin{align*}
    \||u_t(\tau,\cdot)|^p\|_{\dot{H}^{s-2\sigma_2}} &\lesssim \|u_t(\tau,\cdot)\|_{\dot{H}^{s-2\sigma_2}} \|u_t(\tau,\cdot)\|_{L^{\infty}}^{p-1}\\
    &\lesssim \|u_t(\tau,\cdot)\|_{\dot{H}^{s-2\sigma_2}} \left(\|u_t(\tau,\cdot)\|_{\dot{H}^{s^{*}}}+\|u_t(\tau,\cdot)\|_{\dot{H}^{s-2\sigma_2}}\right)^{p-1}.
\end{align*}
After applying the fractional Gagliardo-Nirenberg inequality, it follows
\begin{align*}
\|u_t(\tau,\cdot)\|_{\dot{H}^{s^{*}}} &\lesssim \|u_t(\tau,\cdot)\|_{L^2}^{1-\theta} \| |D|^{s-2\sigma_2}u_t(\tau,\cdot)\|_{L^2}^{\theta} \\
&\lesssim (1+\tau)^{-\frac{n}{2(\sigma-\sigma_1)}(\frac{1}{m}-\frac{1}{2})-1+\frac{\sigma_1}{\sigma-\sigma_1}-\frac{s^{*}}{2(\sigma-\sigma_1)}} \|u\|_{X(\tau)},
\end{align*}
where $\theta = \frac{s^{*}}{s-2\sigma_2}$. Hence, we derive
\begin{align*}
\||u_t(\tau,\cdot)|^p\|_{\dot{H}^{s-2\sigma_2}} &\lesssim (1+\tau)^{p(-\frac{n}{2(\sigma-\sigma_1)}(\frac{1}{m}-\frac{1}{2})-1 +\frac{\sigma_1}{\sigma-\sigma_1})-\frac{s-2\sigma_2}{2(\sigma-\sigma_1)}-(p-1)\frac{s^{*}}{2(\sigma-\sigma_1)}} \|u\|_{X(\tau)}^p \\
&= (1+\tau)^{p(-\frac{n}{2(\sigma-\sigma_1)}(\frac{1}{m}-\frac{1}{2})-1 +\frac{\sigma_1}{\sigma-\sigma_1})-\frac{s-2\sigma_2}{2(\sigma-\sigma_1)}-(p-1)\frac{\frac{n}{2}-\epsilon}{2(\sigma-\sigma_1)}}\|u\|_{X(\tau)}^p \\
&\lesssim (1+\tau)^{p(-\frac{n}{2m(\sigma-\sigma_1)}-1 +\frac{\sigma_1}{\sigma-\sigma_1})+ (p-1)\frac{\epsilon}{2(\sigma-\sigma_1)}}\|u\|_{X(\tau)}^p \\
&\lesssim (1+\tau)^{-\frac{pn}{2(\sigma-\sigma_1)}(\frac{1}{m}-\frac{1}{2p})-\frac{p(\sigma-2\sigma_1)}{\sigma-\sigma_1}}\|u\|_{X(\tau)}^p,
\end{align*}
if we choose $s^{*} = \frac{n}{2}-\epsilon$, where $\epsilon$ is a sufficiently small positive constant. Here we have
\begin{align*}
    \||u_t(\tau,\cdot)|^p\|_{\dot{H}^{s-2\sigma_2}} &\lesssim (1+\tau)^{-\frac{pn}{2(\sigma-\sigma_1)}(\frac{1}{m}-\frac{1}{2p})-\frac{p(\sigma-2\sigma_1)}{\sigma-\sigma_1}}\|u\|_{X(\tau)}^p \\
    &\lesssim(1+\tau)^{-\frac{n}{2m(\sigma-\sigma_1)}(p-1)-\frac{p(\sigma-2\sigma_1)}{\sigma-\sigma_1}}\|u\|_{X(\tau)}^p.
\end{align*}
For this reason, combining inequality \eqref{equation3.3.1}, we can conclude the following estimates:
\begin{align}
    &\int_0^{\frac{t}{2}} (1+t-\tau)^{-\frac{n}{2(\sigma-\sigma_1)}(\frac{1}{m} -\frac{1}{2})-\frac{s}{2(\sigma-\sigma_1)}+\frac{\sigma_1}{\sigma-\sigma_1}} \| |u_t(\tau,\cdot)|^p \|_{L^m \cap L^2 \cap \dot{H}^{s-2\sigma_2}} d\tau \notag\\
    &\qquad\quad\lesssim (1+t)^{-\frac{n}{2(\sigma-\sigma_1)}(\frac{1}{m} -\frac{1}{2})-\frac{s}{2(\sigma-\sigma_1)}+\frac{\sigma_1}{\sigma-\sigma_1}} \|u\|_{X(t)}^p\notag\\
    &\qquad\quad\lesssim (1+t)^{-\frac{\sigma_2}{\sigma-\sigma_1}} \|u\|_{X(t)}^p.\label{Main.estimate3.3.1}
\end{align}
On the other hand, we also have it
\begin{align}
    &\int_{\frac{t}{2}}^t (1+t-\tau)^{-\frac{s}{2(\sigma-\sigma_1)}+\frac{\sigma_1}{\sigma-\sigma_1}} \| |u_t(\tau,\cdot)|^p \|_{L^2 \cap \dot{H}^{s-2\sigma_2}} d\tau \notag\\&\quad\lesssim 
    \begin{cases}
    \|u\|_{X(t)}^p(1+t)^{-\frac{pn}{2(\sigma-\sigma_1)}(\frac{1}{m}-\frac{1}{2p})-\frac{p(\sigma-2\sigma_1)}{\sigma-\sigma_1}} \left((1+t)^{1+\frac{\sigma_1}{\sigma-\sigma_1}-\frac{s}{2(\sigma-\sigma_1)}}+1\right)
    &\text{ if } s \ne 2\sigma,\\
    \|u\|_{X(t)}^p(1+t)^{-\frac{pn}{2(\sigma-\sigma_1)}(\frac{1}{m}-\frac{1}{2p})-\frac{p(\sigma-2\sigma_1)}{\sigma-\sigma_1}} \log(1+t) &\text{ if } s = 2\sigma.
    \end{cases}\notag\\
    &\lesssim (1+t)^{-\frac{\sigma_2}{\sigma-\sigma_1}} \|u\|_{X(t)}^p.\label{Main.estimate3.3.2}
\end{align}
The estimate \eqref{Main.estimate3.3.2} is obtained from $s > 2\sigma_2+\frac{n}{2} > 2(\sigma_2+\sigma_1)$ imply
\begin{align*}
    &-\frac{pn}{2(\sigma-\sigma_1)}(\frac{1}{m}-\frac{1}{2p})-\frac{p(\sigma-2\sigma_1)}{\sigma-\sigma_1}+1+\frac{\sigma_1}{\sigma-\sigma_1}-\frac{s}{2(\sigma-\sigma_1)} < -\frac{\sigma_2}{\sigma-\sigma_1}\\
    \text{ and }&\\
    &-\frac{pn}{2(\sigma-\sigma_1)}(\frac{1}{m}-\frac{1}{2p})-\frac{p(\sigma-2\sigma_1)}{\sigma-\sigma_1} < -\frac{\sigma_2}{\sigma-\sigma_1}.
\end{align*}
From \eqref{Main.estimate3.3.1} and \eqref{Main.estimate3.3.2}, we have the conclusion
\begin{align}
    \||D|^su^{\rm non}(t,\cdot)\|_{L^2} \lesssim (1+t)^{-\frac{\sigma_2}{\sigma-\sigma_1}}\|u\|_{X(t)}^p. \label{3.3.2}
\end{align}
Finally, we estimate the norms $\|u^{\rm non}_t(t,\cdot)\|_{L^2}$ and $\||D|^{s-2\sigma_2}u_t^{\rm non}(t,\cdot)\|_{L^2}$. We have
\begin{align}
    \|u_t^{\rm non}(t,\cdot)\|_{L^2} \lesssim& (1+t)^{-\frac{n}{2(\sigma-\sigma_1)}(\frac{1}{m}-\frac{1}{2})-1+ \frac{\sigma_1}{\sigma-\sigma_1}}\|u\|_{X(t)}^p \int_0^{\frac{t}{2}}(1+\tau)^{-\frac{n}{2m(\sigma-\sigma_1)}(p-1)-\frac{p(\sigma-2\sigma_1)}{\sigma-\sigma_1}}d\tau \notag \\
    &+ (1+t)^{-\frac{n}{2m(\sigma-\sigma_1)}(p-1)-\frac{p(\sigma-2\sigma_1)}{\sigma-\sigma_1}}\|u\|_{X(t)}^p \int_{\frac{t}{2}}^t (1+t-\tau)^{-\frac{n}{2(\sigma-\sigma_1)}(\frac{1}{m}-\frac{1}{2})-1+ \frac{\sigma_1}{\sigma-\sigma_1}} d\tau\notag\\
    \lesssim& (1+t)^{-\frac{n}{2(\sigma-\sigma_1)}(\frac{1}{m}-\frac{1}{2})-1+\frac{\sigma_1}{\sigma-\sigma_1}} \|u\|_{X(t)}^p \label{3.3.3}
    \end{align}
and
    \begin{align}
\||D|^{s-2\sigma_2} u_t^{\rm non}(t,\cdot)\|_{L^2} \lesssim& (1+t)^{-\frac{n}{2(\sigma-\sigma_1)}(\frac{1}{m}-\frac{1}{2})-1+ \frac{\sigma_1}{\sigma-\sigma_1}-\frac{s-2\sigma_2}{2(\sigma-\sigma_1)}}\|u\|_{X(t)}^p \int_0^{\frac{t}{2}}\| |u_t(\tau,\cdot)|^p\|_{L^m \cap L^2 \cap \dot{H}^{s-2\sigma_2}}d\tau \notag\\
    &+ \|u\|_{X(t)}^p\int_{\frac{t}{2}}^t (1+t-\tau)^{-\frac{s-2\sigma_2}{2(\sigma-\sigma_1)}}\||u_t(\tau,\cdot)|^p\|_{L^2 \cap \dot{H}^{s-2\sigma_2}}d\tau \notag\\
\lesssim& (1+t)^{-\frac{n}{2(\sigma-\sigma_1)}(\frac{1}{m}-\frac{1}{2})-1+\frac{\sigma_1}{\sigma-\sigma_1}-\frac{s-2\sigma_2}{2(\sigma-\sigma_1)}}\|u\|_{X(t)}^p.\label{3.3.4}
\end{align}
The estimate \eqref{3.3.3} occurs because $p > 1+ s-2\sigma_2$ implies that
\begin{align*}
 -\frac{n}{2m(\sigma-\sigma_1)}(p-1)-\frac{p(\sigma-2\sigma_1)}{\sigma-\sigma_1} < -\frac{n}{2(\sigma-\sigma_1)}\left(\frac{1}{m}-\frac{1}{2}\right)-1+\frac{\sigma_1}{\sigma-\sigma_1}.
\end{align*}
The estimate \eqref{3.3.4} occurs because the second inequality of \eqref{equation3.3.1} and $s > 2(\sigma+\sigma_2-\sigma_1)$ imply that
\begin{align*}
    &\int_{\frac{t}{2}}^t (1+t-\tau)^{-\frac{s-2\sigma_2}{2(\sigma-\sigma_1)}}\||u_t(\tau,\cdot)|^p\|_{L^2 \cap \dot{H}^{s-2\sigma_2}}d\tau\\
    &\qquad\quad \lesssim (1+t)^{-\frac{pn}{2(\sigma-\sigma_1)}(\frac{1}{m}-\frac{1}{2p})-\frac{p(\sigma-2\sigma_1)}{\sigma-\sigma_1}} \left(1+(1+t)^{1-\frac{s-2\sigma_2}{2(\sigma-\sigma_1)}}\right)\\
    &\qquad\quad\lesssim  (1+t)^{-\frac{pn}{2(\sigma-\sigma_1)}(\frac{1}{m}-\frac{1}{2p})-\frac{p(\sigma-2\sigma_1)}{\sigma-\sigma_1}}\\
    &\qquad\quad \lesssim (1+t)^{-\frac{n}{2(\sigma-\sigma_1)}(\frac{1}{m}-\frac{1}{2})-1+\frac{\sigma_1}{\sigma-\sigma_1}-\frac{s-2\sigma_2}{2(\sigma-\sigma_1)}}.
\end{align*}
From estimates \eqref{3.3.1}, \eqref{3.3.2}, \eqref{3.3.3} and \eqref{3.3.4}, we can conclude the estimate \eqref{2.100}.

\textit{Next, let us prove the inequality \eqref{2.33}.} Using again the $(L^m \cap L^2)-L^2$ estimates if $\tau \in \left[0, \frac{t}{2}\right]$ and the $L^2-L^2$ estimates if $\tau \in \left[\frac{t}{2}, t\right]$ from Theorem \ref{Linear_Decay}, we derive for two functions $u$ and $v$ from $X(t)$ the following estimates:
\begin{align*}
    \|\mathcal{N}[u](t,\cdot)-\mathcal{N}[v](t,\cdot)\|_{L^2} &\lesssim \int_0^{\frac{t}{2}}(1+t-\tau)^{-\frac{n}{2(\sigma-\sigma_1)}(\frac{1}{m}-\frac{1}{2})+\frac{\sigma_1}{\sigma-\sigma_1}} \| |u_t(\tau,\cdot)|^p-|v_t(\tau,\cdot)|^p\|_{L^m \cap L^2} d\tau\\
    &\quad+ \int_{\frac{t}{2}}^t (1+t-\tau)^{-\frac{n}{2(\sigma-\sigma_1)}(\frac{1}{m}-\frac{1}{2})+\frac{\sigma_1}{\sigma-\sigma_1}}\| |u_t(\tau,\cdot)|^p-|v_t(\tau,\cdot)|^p\|_{L^m \cap L^2} d\tau
\end{align*}
and
\begin{align*}    
    &\||D|^{s}\left(\mathcal{N}[u](t,\cdot)-\mathcal{N}[v](t,\cdot)\right)\|_{L^2} \\
    &\quad \lesssim \int_0^{\frac{t}{2}} (1+t-\tau)^{-\frac{n}{2(\sigma-\sigma_1)}(\frac{1}{m} -\frac{1}{2})-\frac{s}{2(\sigma-\sigma_1)}+\frac{\sigma_1}{\sigma-\sigma_1}} \| |u_t(\tau,\cdot)|^p -|v_t(\tau,\cdot)|^p\|_{L^m \cap L^2 \cap \dot{H}^{s-2\sigma_2}} d\tau\\
    &\qquad + \int_{\frac{t}{2}}^t (1+t-\tau)^{-\frac{s-2\sigma_1}{2(\sigma-\sigma_1)}} \| |u_t(\tau,\cdot)|^p -|v_t(\tau,\cdot)|^p\|_{L^2 \cap \dot{H}^{s-2\sigma_2}} d\tau.
\end{align*}
By using Holder's inequality, we get
\begin{align*}
    \| |u_t(\tau,\cdot)|^p -|v_t(\tau,\cdot)|^p\|_{L^m} &\lesssim \|u_t(\tau,\cdot)-v_t(\tau,\cdot)\|_{L^{mp}} \left(\|u_t(\tau,\cdot)\|_{L^{mp}}^{p-1}+\|v_t(\tau,\cdot)\|_{L^{mp}}^{p-1}\right), \\
    \| |u_t(\tau,\cdot)|^p -|v_t(\tau,\cdot)|^p\|_{L^2} &\lesssim \|u_t(\tau,\cdot)-v_t(\tau,\cdot)\|_{L^{2p}} \left(\|u_t(\tau,\cdot)\|_{L^{2p}}^{p-1}+\|v_t(\tau,\cdot)\|_{L^{2p}}^{p-1}\right).
\end{align*}
Applying the fractional Gagliardo-Nirenberg inequality to the term
\begin{equation*}
    \|u_t(\tau,\cdot)-v_t(\tau,\cdot)\|_{L^h}, \|u_t(\tau,\cdot)\|_{L^h}, \|v_t(\tau,\cdot)\|_{L^h}
\end{equation*}
with $h = 2p \text{ and } h = mp$, we have
\begin{align*}
     \| |u_t(\tau,\cdot)|^p -|v_t(\tau,\cdot)|^p\|_{L^{m}} &\lesssim (1+\tau)^{-\frac{n}{2m(\sigma-\sigma_1)}(p-1)-\frac{p(\sigma-2\sigma_1)}{\sigma-\sigma_1}} \|u-v\|_{X(t)}\left(\|u\|_{X(t)}^{p-1}+\|v\|_{X(t)}^{p-1}\right), \\
     \| |u_t(\tau,\cdot)|^p -|v_t(\tau,\cdot)|^p\|_{L^{2}} &\lesssim (1+\tau)^{-\frac{pn}{2(\sigma-\sigma_1)}(\frac{1}{m}-\frac{1}{2p})-\frac{p(\sigma-2\sigma_1)}{\sigma-\sigma_1}} \|u-v\|_{X(t)}\left(\|u\|_{X(t)}^{p-1}+\|v\|_{X(t)}^{p-1}\right).
\end{align*}
We now focus our attention on estimates $\| |u_t(\tau,\cdot)|^p -|v_t(\tau,\cdot)|^p\|_{ \dot{H}^{s-2\sigma_2}}$. By using the integral representation 
\begin{equation*}
    |u_t(\tau,\cdot)|^p -|v_t(\tau,\cdot)|^p = p \int_0^1\left(u_t(\tau,x)-v_t(\tau,x)\right)G\left(\omega u_t(\tau,x) + (1-\omega)v_t(\tau,x)\right) d\omega,
\end{equation*}
where $G(u)= u|u|^{p-2}$, we obtain
\begin{align*}
    &\| |u_t(\tau,\cdot)|^p -|v_t(\tau,\cdot)|^p\|_{ \dot{H}^{s-2\sigma_2}} \\
    &\qquad \lesssim \int_0^1 \left\||D|^{s-2\sigma_2}\left(\left(u_t(\tau,x)-v_t(\tau,x)\right)G\left(\omega u_t(\tau,x) + (1-\omega)v_t(\tau,x)\right)\right)\right\|_{L^2}d\omega.
\end{align*}
Thanks to the fractional powers rule, we can proceed as follows
\begin{align*}
    \| |u_t(\tau,\cdot)|^p -|v_t(\tau,\cdot)|^p\|_{ \dot{H}^{s-2\sigma_2}} &\lesssim \int_0^1 \|u_t(\tau,\cdot)-v_t(\tau,\cdot)\|_{\dot{H}^{s-2\sigma_2}}\|G(\omega u_t(\tau,\cdot)+(1-\omega)v_t(\tau,\cdot))\|_{L^{\infty}}d\omega \\
    &\quad +\int_0^1 \|u_t(\tau,\cdot)-v_t(\tau,\cdot)\|_{L^{\infty}}\|G(\omega u_t(\tau,\cdot)+(1-\omega)v_t(\tau,\cdot))\|_{\dot{H}^{s-2\sigma_2}}d\omega\\
    &\lesssim \int_0^1 \|u_t(\tau,\cdot)-v_t(\tau,\cdot)\|_{\dot{H}^{s-2\sigma_2}}\|\omega u_t(\tau,\cdot)+(1-\omega)v_t(\tau,\cdot)\|_{L^{\infty}}^{p-1}d\omega \\
    &\quad+\int_0^1 \|u_t(\tau,\cdot)-v_t(\tau,\cdot)\|_{L^{\infty}}\|G(\omega u_t(\tau,\cdot)+(1-\omega)v_t(\tau,\cdot))\|_{\dot{H}^{s-2\sigma_2}}d\omega.
\end{align*}
Applying Proposition \ref{FractionalPowers} with $p > 2$ and $s-2\sigma_2 \in \left(\frac{n}{2}, p-1\right)$, we get
\begin{align*}
    & \|G(\omega u_t(\tau,\cdot)+(1-\omega)v_t(\tau,\cdot))\|_{\dot{H}^{s-2\sigma_2}} \\
    &\qquad \lesssim \|\omega u_t(\tau,\cdot)+(1-\omega)v_t(\tau,\cdot)\|_{\dot{H}^{s-2\sigma_2}} \|\omega u_t(\tau,\cdot)+(1-\omega)v_t(\tau,\cdot)\|_{L^{\infty}}^{p-2}.
\end{align*}
Using Proposition \ref{Embedding} with a suitable $s^{*} < \frac{n}{2}$, we obtain
\begin{equation*}
    \|u_t(\tau,\cdot)-v_t(\tau,\cdot)\|_{L^{\infty}} \lesssim \|u_t(\tau,\cdot)-v_t(\tau,\cdot)\|_{\dot{H}^{s^{*}}} + \|u_t(\tau,\cdot)-v_t(\tau,\cdot)\|_{\dot{H}^{s-2\sigma_2}}.
\end{equation*}
Applying the fractional Gagliardo-Nirenberg inequality, we have
\begin{align*}
    \|u_t(\tau,\cdot)-v_t(\tau,\cdot)\|_{\dot{H}^{s^{*}}} &\lesssim \|u_t(\tau,\cdot)-v_t(\tau,\cdot)\|_{L^2}^{1-\theta}\| |D|^{s-2\sigma_2}\left(u_t(\tau,\cdot)-v_t(\tau,\cdot)\right)\|_{L^2}^{\theta} \\
    &\lesssim (1+\tau)^{-\frac{n}{2(\sigma-\sigma_1)}(\frac{1}{m}-\frac{1}{2})-1+\frac{\sigma_1}{\sigma-\sigma_1}-\frac{s^{*}}{2(\sigma-\sigma_1)}} \|u-v\|_{X(\tau)},
\end{align*}
where $\theta = \frac{s^{*}}{s-2\sigma_2}$. In the same way, we get
\begin{equation*}
    \|\omega u_t(\tau,\cdot)+(1-\omega)v_t(\tau,\cdot)\|_{L^{\infty}} \lesssim (1+\tau)^{-\frac{n}{2(\sigma-\sigma_1)}\left(\frac{1}{m}-\frac{1}{2}\right)-1+\frac{\sigma_1}{\sigma-\sigma_1}-\frac{s^{*}}{2(\sigma-\sigma_1)}} \|\omega u +(1-\omega)v\|_{X(\tau)}.
\end{equation*}
Therefore, we may conclude
\begin{align*}
     &\| |u_t(\tau,\cdot)|^p -|v_t(\tau,\cdot)|^p\|_{ \dot{H}^{s-2\sigma_2}} \\
     &\quad\lesssim (1+\tau)^{p(-\frac{n}{2(\sigma-\sigma_1)}(\frac{1}{m}-\frac{1}{2})-1+\frac{\sigma_1}{\sigma-\sigma_1})-\frac{s-2\sigma_2}{2(\sigma-\sigma_1)}-(p-1)\frac{s^{*}}{2(\sigma-\sigma_1)}} \|u-v\|_{X(t)} \left(\|u\|_{X(t)}^{p-1}+\|v\|_{X(t)}^{p-1}\right).
\end{align*}
Then, repeating the steps in the proof of estimate \eqref{2.100}, we have estimate \eqref{2.33}. Summarizing, the proof of Theorem \ref{Global_Existence-3} is completed.
\end{proof}

%=====================================================================
\subsection{Proof of asymptotic profiles}\label{Semi-LinearAsymptotic.Sec}
In this part, we will prove the estimate in Theorem \ref{Semi-linear_Asym}. Firstly, we consider the following Cauchy problem:
\begin{equation} \label{Heat-Problem}
\begin{cases}
v_t + (-\Delta)^{\sigma-\sigma_1}v = 0, &\quad x\in \R^n,\, t \ge 0, \\
v(0,x) = v_0(x), &\quad x\in \R^n,
\end{cases}
\end{equation}
and we choose the suitable data $v_0 = u_0 + I_{\sigma_1} u_1$ where $I_{\sigma_1}$ is the Riesz potential.
We can easily find that the solutions to \eqref{Heat-Problem} are written by the following formula:
\begin{equation}\label{eq:formula-v}
    v(t,x) = \mathcal{G}_0 \ast v_0 = \mathcal{G}_0 \ast (u_0 + I_{\sigma_1} \ast u_1) = \mathcal{G}_0 \ast u_0 + \mathcal{G}_1 \ast u_1.
\end{equation}

Next, we introduce some auxiliary propositions for our proof.

\begin{proposition}\label{proposition:Asy2}
    The Sobolev solutions to \eqref{Linear_Main.Eq} satisfy the following estimates for $j = 0, 1$:
    \begin{align}\label{eq:proposition-asy2}
        \left\|\partial^j_t |D|^s (u(t, \cdot) - v(t, \cdot))\right\|_{L^2} 
        &\lesssim t^{\Gamma(s)} \|(u_0, u_1)\|_{L^1} + e^{-t} \|u_0\|_{H^{s+2j\sigma_2} \cap H^{s+2j(\sigma-\sigma_1)}} \nonumber\\ 
        &\quad + e^{-t} \|u_1\|_{H^{[s+2(j-1)\sigma_2]^+} \cap H^{[s+2j(\sigma-\sigma_1)-2\sigma_1]^+}},
    \end{align}
    for large $t \geq 1$, any  $s \geq 0$ and
    \begin{equation*}
        \Gamma(s) \coloneqq 
        \begin{cases}
            \max \left\{-\frac{n}{4(\sigma-\sigma_1)}-\frac{s}{2(\sigma-\sigma_1)}-j-\frac{\sigma-3\sigma_1}{\sigma-\sigma_1}, -\frac{n}{4\sigma_1}- \frac{s}{2\sigma_1}-j+1\right\} &\text{ if } \sigma_1 + \sigma_2 > \sigma,\\[0.5em]
            \max \left\{-\frac{n}{4(\sigma-\sigma_1)}-\frac{s}{2(\sigma-\sigma_1)}-j-\frac{\sigma_2-2\sigma_1}{\sigma-\sigma_1}, -\frac{n}{4\sigma_1}- \frac{s}{2\sigma_1}-j+1\right\} &\text{ if } \sigma_1 + \sigma_2 \leq \sigma.
        \end{cases}
    \end{equation*}
\end{proposition}
\begin{proof}
    We divide the left-hand term of \eqref{eq:proposition-asy2} into parts as follows:
    \begin{align*}
    \left\|\partial_t^j |D|^s \Big(u(t,\cdot)- v(t,\cdot) \Big)\right\|_{L^2}
    & \lesssim 
    \left\|\partial_t^j |D|^s \bigg(\big(K^1_0(t,x) - \mathcal{G}_0(t,x)\big) \ast u_0(x)\bigg)(t,\cdot) \right\|_{L^2}\\
    &\quad + \left\|\partial_t^j |D|^s \bigg(\big(K^1_1(t,x) - \mathcal{G}_1(t,x)\big) \ast u_1(x) \bigg)(t,\cdot) \right\|_{L^2}\\
    &\quad + \left\|\partial_t^j |D|^s \big(K^2_0(t,x) \ast u_0(t,x)\big)(t,\cdot) \right\|_{L^2}\\
    &\quad + \left\|\partial_t^j |D|^s \big(K^2_1(t,x) \ast u_1(t,x)\big)(t,\cdot) \right\|_{L^2}\\
    & \coloneqq J_1 + J_2 + J_3 + J_4.
    \end{align*}
    For $J_1$ and $J_2$, we use estimates \eqref{proposition3.6.1} and \eqref{proposition3.6.2} in Proposition \ref{proposition3.6}, respectively. For $J_3$, we combine estimates \eqref{lemma2.6.2} and \eqref{lemma2.6.5}. For $J_4$, we combine estimates \eqref{lemma2.6.4} and \eqref{lemma2.6.6}. Then, we can conclude the desired estimates.
\end{proof}

\begin{proposition}\label{proposition:Asy3}
    We have the following estimate for any $t \geq 0$, any $s \geq 0$ and $j=0,1$:
    \begin{align}\label{eq:proposition-asy3-2}
        \left\|\partial^j_t |D|^s \mathcal{G}_1(t, \cdot) \right\|_{L^2} 
        \lesssim t^{-\frac{n}{4(\sigma-\sigma_1)} - \frac{s}{2(\sigma-\sigma_1)} - j + \frac{\sigma_1}{\sigma-\sigma_1}}.
    \end{align}
    Moreover, the Sobolev solutions to \eqref{Heat-Problem} satisfy the following estimates for large $t \geq 1$, any $s \geq 0$ and $j=0,1$
    \begin{align}\label{eq:proposition-asy3-1}
        \left\|\partial^j_t |D|^s v(t, \cdot) \right\|_{L^2} 
        \lesssim t^{-\frac{n}{4(\sigma-\sigma_1)} - \frac{s}{2(\sigma-\sigma_1)} - j + \frac{\sigma_1}{\sigma-\sigma_1}} \|(u_0, u_1)\|_{L^1}.
    \end{align}
\end{proposition}
\begin{proof}
    From \eqref{the1.1.1}, we immediately have the estimate \eqref{eq:proposition-asy3-2}.
    In order to prove \eqref{eq:proposition-asy3-1}, we use the formula \eqref{eq:formula-v} and the Young inequality to have: 
    \begin{align*}
        \left\|\partial^j_t |D|^s v(t, \cdot) \right\|_{L^2} 
        &\lesssim \left\|\partial^j_t |D|^s \big(\mathcal{G}_0(t,x) \ast u_0(t,x)\big)(t, \cdot) \right\|_{L^2} 
        + \left\|\partial^j_t |D|^s \big(\mathcal{G}_1(t,x) \ast u_1(t,x)\big)(t, \cdot) \right\|_{L^2}\\
        &\lesssim \left\|\partial^j_t |D|^s \mathcal{G}_0(t, \cdot) \right\|_{L^2} \| u_0(t, \cdot)\|_{L^1}
        + \left\|\partial^j_t |D|^s \mathcal{G}_1(t, \cdot) \right\|_{L^2} \| u_1(t, \cdot)\|_{L^1}.
    \end{align*}
    We repeat some arguments in the proof of \eqref{eq:proposition-asy3-2} for each $I_1$ and $I_2$ to conclude \eqref{eq:proposition-asy3-1}. Hence, the proof of Proposition \ref{proposition:Asy3} is complete.
\end{proof}

\begin{proof}[\textbf{Proof of Theorem \ref{Semi-linear_Asym}}]
Firstly, we will prove in detail the estimate \eqref{eq:Semi-linear_Asym}. By using the estimate in Theorem \ref{Linear_Asym} and recalling the formula \eqref{2.31} of the solution to \eqref{Main.Eq}, we can reduce the estimate \eqref{eq:Semi-linear_Asym} to the following desired estimate:
\begin{equation*}
    \left\|\partial^j_t |D|^{2k\sigma_2} \left(\int_0^t K_1(t-\tau,x) \ast_x |u(\tau,x)|^p d\tau - M_0 \mathcal{G}_1(t,x)\right)\right\|_{L^2} 
    = o\left(t^{-\frac{n}{4(\sigma-\sigma_1)} - \frac{k\sigma_2}{\sigma-\sigma_1} - j + \frac{\sigma_1}{\sigma-\sigma_1}}\right),
\end{equation*}
and we can re-write the above estimate to
\begin{align}
    &\left\|\int_0^t \partial^j_t |D|^{2k\sigma_2}(K_1(t-\tau,x) \ast_x |u(\tau,x)|^p) d\tau - M_0 \partial^j_t |D|^{2k\sigma_2} \mathcal{G}_1(t,x) \right\|_{L^2} \nonumber \\
    &\qquad = o\left(t^{-\frac{n}{4(\sigma-\sigma_1)} - \frac{k\sigma_2}{\sigma-\sigma_1} - j + \frac{\sigma_1}{\sigma-\sigma_1}}\right), \label{eq:need-proof}
\end{align}
because of the fact that $K_1(0,x) \ast |u(t,x)|^p = 0$. Now we divide the left-hand side term of \eqref{eq:need-proof} in the $L^2$ norm into five parts as follows:
\begin{align*}
    \int_0^t \partial^j_t |D|^{2k\sigma_2}(K_1(t-\tau,x) &\ast_x |u(\tau,x)|^p) d\tau - M \partial^j_t |D|^{2k\sigma_2} \mathcal{G}_1\\
    &=\int_0^{\frac{t}{2}} \partial^j_t |D|^{2k\sigma_2} \left((K_1(t-\tau,x) - \mathcal{G}_1(t-\tau,x))\ast_x |u(\tau,x)|^p\right) d\tau\\
    &\quad +\int_{\frac{t}{2}}^{t} \partial^j_t |D|^{2k\sigma_2} (K_1(t-\tau,x) \ast_x |u(\tau,x)|^p) d\tau\\
    &\quad +\int_0^{\frac{t}{2}} \partial^j_t |D|^{2k\sigma_2} \left((\mathcal{G}_1(t-\tau,x) - \mathcal{G}_1(t,x))\ast_x |u(\tau,x)|^p\right) d\tau\\
    &\quad +\int_0^{\frac{t}{2}} \partial^j_t |D|^{2k\sigma_2} \left(\mathcal{G}_1(t,x) \ast_x |u(\tau,x)|^p - \left(\int_{\mathbb{R}^n} |u(\tau,y)|^p dy\right) \mathcal{G}_1(t,x)\right) d\tau\\
    &\quad -\left(\int^{\infty}_{\frac{t}{2}} \int_{\mathbb{R}^n} |u(\tau, y)|^p dy d\tau \right) \partial^j_t |D|^{2k\sigma_2} \mathcal{G}_1(t,x)\\
    &\coloneqq I_1 + I_2 + I_3 + I_4 + I_5.
\end{align*}
Let us deal with the integral $I_1$.  We have
\begin{align}
    \|I_1\|_{L^2} 
    &\lesssim \int_0^{\frac{t}{2}} \|\partial^j_t |D|^{2k\sigma_2} \left((K_1(t-\tau,x) - \mathcal{G}_1(t-\tau,x))\ast_x |u(\tau,x)|^p\right)\|_{L^2} d\tau \nonumber \\ 
    &\lesssim \int_0^{\frac{t}{2}} (t-\tau)^{\Gamma(2k\sigma_2)} \||u(\tau,\cdot)|^p\|_{L^1} d\tau 
    + \int_0^{\frac{t}{2}} e^{-(t-\tau)} \||u(\tau,\cdot)|^p\|_{L^2} d\tau  \label{eq:I1-1}\\
    &\lesssim t^{\Gamma(2k\sigma_2)} \int_0^{\frac{t}{2}} (1+\tau)^{-\frac{n}{2(\sigma-\sigma_1)}(p-1)+\frac{p\sigma_1}{\sigma-\sigma_1}} d\tau 
    + e^{-\frac{t}{2}}\int_0^{\frac{t}{2}}  (1+\tau)^{-\frac{n}{2(\sigma-\sigma_1)}(p-\frac{1}{2})+\frac{p\sigma_1}{\sigma-\sigma_1}} d\tau \label{eq:I1-2}\\
    &\lesssim t^{-\frac{n}{4(\sigma-\sigma_1)} - \frac{k\sigma_2}{\sigma-\sigma_1} - j + \frac{\sigma_1}{\sigma-\sigma_1}} + e^{-\frac{t}{2}}
    \lesssim t^{-\frac{n}{4(\sigma-\sigma_1)} - \frac{k\sigma_2}{\sigma-\sigma_1} - j + \frac{\sigma_1}{\sigma-\sigma_1}} \label{eq:I1-3}.
\end{align}
Here, we applied the estimate \eqref{eq:proposition-asy2} in \eqref{eq:I1-1}. We also used estimates \eqref{2.37} and \eqref{2.38} combining with the relation $t-\tau \approx t$ when $\tau \in [0, t/2]$ in \eqref{eq:I1-2}. Furthermore, the two above integrals in \eqref{eq:I1-2} are integrable under the condition \eqref{2.28-1}. Besides that, we also always have the following inequality when $n > 4\sigma_1$:
\begin{equation*}
    \Gamma(2k\sigma_2) < -\frac{n}{4(\sigma-\sigma_1)} - \frac{k\sigma_2}{\sigma-\sigma_1} - j + \frac{\sigma_1}{\sigma-\sigma_1}.
\end{equation*}
We control $I_2$ by repeating some arguments in the proof of Theorem \ref{Global_Existence-1} and we obtain
\begin{align}
    \|I_2\|_{L^2} 
    \lesssim t^{-\frac{n}{4(\sigma-\sigma_1)} - \frac{k\sigma_2}{\sigma-\sigma_1} - j + \frac{\sigma_1}{\sigma-\sigma_1} -\varepsilon} \label{eq:I2},
\end{align}
where $\varepsilon$ is a sufficiently small positive satisfying $\frac{n}{2(\sigma-\sigma_1)}(p-1) - \frac{p\sigma_1}{\sigma-\sigma_1} - 1 > \varepsilon$. To take account of $I_3$, we use the mean value theorem on $t$ to get the following representation:
\begin{equation*}
    \mathcal{G}_1(t-\tau,x) - \mathcal{G}_1(t,x) = -\tau \partial_t \mathcal{G}_1(t-\omega_1\tau, x),
\end{equation*}
with $\omega_1$ is a constant in $[0,1]$. From that, we have
\begin{equation}
    \|I_3\|_{L^2} \lesssim \int_0^{\frac{t}{2}} \tau \|\partial_t^{j+1}|D|^{2k\sigma_2}(\mathcal{G}_1(t-\omega_1\tau, x)\ast|u(\tau, x)|^p)\|_{L^2} d\tau.
\end{equation}
Applying the estimates \eqref{2.37}, \eqref{eq:proposition-asy3-1} and the relation $t-\omega_1\tau \approx t$ when $\tau \in [0, t/2]$, we have
\begin{align}
    \|I_3\|_{L^2} 
    &\lesssim \int_0^{\frac{t}{2}} \tau (t-\omega_1\tau)^{-\frac{n}{4(\sigma-\sigma_1)} - \frac{k\sigma_2}{\sigma-\sigma_1} - j + \frac{\sigma_1}{\sigma-\sigma_1}-1}\||u(\tau, \cdot)|^p\|_{L^1} d\tau \nonumber\\
    &\lesssim t^{-\frac{n}{4(\sigma-\sigma_1)} - \frac{k\sigma_2}{\sigma-\sigma_1} - j + \frac{\sigma_1}{\sigma-\sigma_1}-1} \int_0^{\frac{t}{2}} \tau (1+\tau)^{-\frac{n}{2(\sigma-\sigma_1)}(p-1)+\frac{p\sigma_1}{\sigma-\sigma_1}} d\tau \nonumber\\
    &\lesssim t^{-\frac{n}{4(\sigma-\sigma_1)} - \frac{k\sigma_2}{\sigma-\sigma_1} - j + \frac{\sigma_1}{\sigma-\sigma_1}-1} \int_0^{\frac{t}{2}} (1+\tau)^{-\frac{n}{2(\sigma-\sigma_1)}(p-1)+\frac{p\sigma_1}{\sigma-\sigma_1}+1} d\tau \nonumber\\
    &\lesssim t^{-\frac{n}{4(\sigma-\sigma_1)} - \frac{k\sigma_2}{\sigma-\sigma_1} - j + \frac{\sigma_1}{\sigma-\sigma_1}-1} 
    \left(1+(1+\tau)^{-\frac{n}{2(\sigma-\sigma_1)}(p-1)+\frac{p\sigma_1}{\sigma-\sigma_1}+2}+\log(1+t)\right) \nonumber\\
    &\lesssim t^{-\frac{n}{4(\sigma-\sigma_1)} - \frac{k\sigma_2}{\sigma-\sigma_1} - j + \frac{\sigma_1}{\sigma-\sigma_1}-1} 
    (1+t)^{-\varepsilon+1}
    \lesssim t^{-\frac{n}{4(\sigma-\sigma_1)} - \frac{k\sigma_2}{\sigma-\sigma_1} - j + \frac{\sigma_1}{\sigma-\sigma_1}-\varepsilon} \label{eq:I3}
\end{align}
as $t \rightarrow \infty$ and $\varepsilon$ is a sufficiently small positive satisfying
\begin{equation*}
    \frac{n}{2(\sigma-\sigma_1)}(p-1) - \frac{p\sigma_1}{\sigma-\sigma_1} - 1 > \varepsilon.
\end{equation*}

Let us estimate $I_4$. Firstly, we separate it into two terms as follows:
\begin{align*}
    I_4 &= \int_0^{\frac{t}{2}} \partial^j_t |D|^{2k\sigma_2} \left(\int_{\mathbb{R}^n} \mathcal{G}_1(t, x-y) |u(\tau, y)|^p dy - \int_{\mathbb{R}^n} \mathcal{G}_1(t,x)|u(\tau,x)|^p dy\right) d\tau\\
    &= \int_0^{\frac{t}{2}} \int_{|y| \leq t^{\frac{1}{4(\sigma-\sigma_1)}}} \partial^j_t |D|^{2k\sigma_2}(\mathcal{G}_1(t, x-y)-\mathcal{G}_1(t,x)) |u(\tau, y)|^p dy d\tau\\
    &\quad+ \int_0^{\frac{t}{2}} \int_{|y| \geq t^{\frac{1}{4(\sigma-\sigma_1)}}} \partial^j_t |D|^{2k\sigma_2}(\mathcal{G}_1(t, x-y)-\mathcal{G}_1(t,x)) |u(\tau, y)|^p dy d\tau\\
    &\coloneqq I_{41} + I_{42}.
\end{align*}
Then, we deal with $I_{41}$ by using the mean value theorem on $x$ to obtain
\begin{equation*}
    \mathcal{G}_1(t,x-y) - \mathcal{G}_1(t,x) = -y \partial_x \mathcal{G}_1(t, x-\omega_2 y),
\end{equation*}
with $\omega_2$ is a constant in $[0,1]$. Hence, we can use \eqref{2.37} and \eqref{eq:proposition-asy3-2} to have
\begin{align*}
    \|I_{41}\|_{L^2} 
    &\lesssim \|\int_0^{\frac{t}{2}} \int_{|y| \leq t^{\frac{1}{4(\sigma-\sigma_1)}}} (-y) \partial^j_t |D|^{2k\sigma_2+1} \mathcal{G}_1(t, x-\omega_2 y) |u(\tau, y)|^p dy d\tau \|_{L^2}\\
    &\lesssim \int_0^{\frac{t}{2}} \int_{|y| \leq t^{\frac{1}{4(\sigma-\sigma_1)}}} |y| \|\partial^j_t |D|^{2k\sigma_2+1} \mathcal{G}_1(t, x-\omega_2 y)\|_{L^2} |u(\tau, y)|^p dy d\tau\\
    &\lesssim t^{-\frac{n}{4(\sigma-\sigma_1)} - \frac{2k\sigma_2+1}{2(\sigma-\sigma_1)} - j + \frac{\sigma_1}{\sigma-\sigma_1}} \int_0^{\frac{t}{2}} \int_{|y| \leq t^{\frac{1}{4(\sigma-\sigma_1)}}} |y||u(\tau, y)|^p dy d\tau\\
    &\lesssim t^{-\frac{n}{4(\sigma-\sigma_1)} - \frac{2k\sigma_2+1}{2(\sigma-\sigma_1)} - j + \frac{\sigma_1}{\sigma-\sigma_1}} \int_0^{\frac{t}{2}} t^{\frac{1}{4(\sigma-\sigma_1)}} \||u(\tau, y)|^p\|_{L^1} d\tau\\
    &\lesssim t^{-\frac{n}{4(\sigma-\sigma_1)} - \frac{k\sigma_2}{\sigma-\sigma_1} - j + \frac{\sigma_1}{\sigma-\sigma_1} - \frac{1}{4(\sigma-\sigma_1)}} \int_0^{\frac{t}{2}} (1+\tau)^{-\frac{n}{2(\sigma-\sigma_1)}(p-1)+\frac{p\sigma_1}{\sigma-\sigma_1}} d\tau\\
    &\lesssim t^{-\frac{n}{4(\sigma-\sigma_1)} - \frac{k\sigma_2}{\sigma-\sigma_1} - j + \frac{\sigma_1}{\sigma-\sigma_1} - \frac{1}{4(\sigma-\sigma_1)}}.
\end{align*}
For $I_{42}$, by \eqref{2.37} we have the following estimate
\begin{equation*}
    \int_0^{\infty} \int_{\mathbb{R}^n} |u(\tau, y)|^p dy d\tau = \int_0^{\infty} \||u(\tau, y)|^p\|_{L^1} d\tau \lesssim \int_0^{\infty} (1+\tau)^{-\frac{n}{2(\sigma-\sigma_1)}(p-1)+\frac{p\sigma_1}{\sigma-\sigma_1}} d\tau \lesssim 1.
\end{equation*}
It leads to the relation
\begin{equation*}
    \lim_{t \rightarrow \infty} \int_0^{\infty} \int_{|y| \geq t^{\frac{1}{4(\sigma-\sigma_1)}}} |u(\tau, y)|^p dy d\tau = 0.
\end{equation*}
From that and \eqref{eq:proposition-asy3-2}, we can estimate $I_{42}$ as follows:
\begin{align*}
    \|I_{42}\|_{L^2} &\leq 2 \int_0^{\frac{t}{2}} \int_{|y| \geq t^{\frac{1}{4(\sigma-\sigma_1)}}} \|\partial^j_t |D|^{2k\sigma_2} \mathcal{G}_1(t,\cdot)\|_{L^2} |u(\tau, y)|^p dy d\tau\\
    &\lesssim t^{-\frac{n}{4(\sigma-\sigma_1)} - \frac{k\sigma_2}{\sigma-\sigma_1} - j + \frac{\sigma_1}{\sigma-\sigma_1}} \int_0^{\frac{t}{2}} \int_{|y| \geq t^{\frac{1}{4(\sigma-\sigma_1)}}} |u(\tau, y)|^p dy d\tau\\
    &= o\left(t^{-\frac{n}{4(\sigma-\sigma_1)} - \frac{k\sigma_2}{\sigma-\sigma_1} - j + \frac{\sigma_1}{\sigma-\sigma_1}}\right)
\end{align*}
as $t \rightarrow \infty$. Therefore, from two estimates of $I_{41}$ and $I_{42}$, we can obtain
\begin{equation}\label{eq:I4}
    \|I_{4}\|_{L^2} = o\left(t^{-\frac{n}{4(\sigma-\sigma_1)} - \frac{k\sigma_2}{\sigma-\sigma_1} - j + \frac{\sigma_1}{\sigma-\sigma_1}}\right)
\end{equation}
as $t \rightarrow \infty$. Finally, we can control $I_5$ by using \eqref{eq:proposition-asy3-2} as follows:
\begin{align}
    \|I_{5}\|_{L^2} 
    &\lesssim \|\partial^j_t |D|^{2k\sigma_2} \mathcal{G}_1(t,x)\|_{L^2}\int^{\infty}_{\frac{t}{2}} \int_{\mathbb{R}^n} |u(\tau, y)|^p dy d\tau \nonumber \nonumber\\
    &\lesssim t^{-\frac{n}{4(\sigma-\sigma_1)} - \frac{k\sigma_2}{\sigma-\sigma_1} - j + \frac{\sigma_1}{\sigma-\sigma_1}} \int^{\infty}_{\frac{t}{2}} \||u(\tau, y)|^p\|_{L^1} d\tau \nonumber\\
    &\lesssim t^{-\frac{n}{4(\sigma-\sigma_1)} - \frac{k\sigma_2}{\sigma-\sigma_1} - j + \frac{\sigma_1}{\sigma-\sigma_1}} \int^{\infty}_{\frac{t}{2}} (1+\tau)^{-\frac{n}{2(\sigma-\sigma_1)}(p-1)+\frac{p\sigma_1}{\sigma-\sigma_1}} d\tau \nonumber\\
    &\lesssim t^{-\frac{n}{4(\sigma-\sigma_1)} - \frac{k\sigma_2}{\sigma-\sigma_1} - j + \frac{\sigma_1}{\sigma-\sigma_1}} (1+t)^{-\frac{n}{2(\sigma-\sigma_1)}(p-1)+\frac{p\sigma_1}{\sigma-\sigma_1}+1} \nonumber\\
    &\lesssim t^{-\frac{n}{4(\sigma-\sigma_1)} - \frac{k\sigma_2}{\sigma-\sigma_1} - j + \frac{\sigma_1}{\sigma-\sigma_1} - \varepsilon} \label{eq:I5}
\end{align}
as $t \rightarrow \infty$ and $\varepsilon$ is a sufficiently small positive satisfying
\begin{equation*}
    \frac{n}{2(\sigma-\sigma_1)}(p-1) - \frac{p\sigma_1}{\sigma-\sigma_1} - 1 > \varepsilon.
\end{equation*}
Summarizing, we combine \eqref{eq:I1-3}, \eqref{eq:I2}, \eqref{eq:I3}, \eqref{eq:I4} and \eqref{eq:I5} to conclude the desired estimate \eqref{eq:need-proof}. Hence, the estimate \eqref{eq:Semi-linear_Asym} is proven.

Next, we have the following estimate:
\begin{equation*}
    \int_{0}^{\frac{t}{2}} \||u_t(\tau, y)|^p\|_{L^1} d\tau 
    \lesssim \int_{0}^{\frac{t}{2}}(1+\tau)^{-\frac{n}{2(\sigma-\sigma_1)}(p-1) - \frac{p(\sigma-2\sigma_1)}{\sigma-\sigma_1}} d\tau.
\end{equation*}
Under the conditions of $p$ and $n$ in the assumption of Theorem \ref{Global_Existence-3}, we have the condition \eqref{equation3.3.2} satisfied to ensure the above right integral is integrable. Therefore, we use exactly the above arguments in the proof of \eqref{eq:Semi-linear_Asym} and replace suitably $|u|^p$ by $|u_t|^p$ in order to prove and conclude the estimates \eqref{eq:Semi-linear_Asym2-1} and \eqref{eq:Semi-linear_Asym2-2}. In other words, Theorem \ref{Semi-linear_Asym} was proved completed.

\end{proof}

%=====================================================================
%\section{Concluding remarks and Open problems}\label{Concluding.Sec}
%\begin{remark}[\textbf{Coefficients in time}]
    
%\end{remark}

%\begin{remark}[\textbf{Lifespan estimates}]
    
%\end{remark}

%=====================================================================
\section*{Appendix}
\begin{lemma}[see \cite{IkehataTakeda2019}] \label{L^1.Lemma}
Let $a\ge 0$. Let us assume $h= h(x) \in L^1$ and $\phi=\phi(t,x)$ be a smooth function satisfying
$$ \big\||D|^a \phi(t,\cdot)\big\|_{L^2} \lesssim t^{-\alpha} \quad \text{ and }\quad  \big\||D|^{a+1} \phi(t,\cdot)\big\|_{L^2} \lesssim t^{-\alpha-\beta}, $$
for some positive constants $\alpha,\,\beta>0$. Then, the following estimate holds:
$$ \left\||D|^a \left(\phi(t,x) \ast h(x)- \left(\int_{\R^n}h(y)\,dy\right)\phi(t,x)\right)(t,\cdot) \right\|_{L^2}= o\big(t^{-\alpha}\big) \quad \text{ as }t \to \ity, $$
for all space dimensions $n\ge 1$.
\end{lemma}

\begin{proposition}[Fractional Gagliardo-Nirenberg inequality, \cite{ReissigEbert2018}] \label{FractionalG-N}
Let $1<p,\, p_0,\, p_1<\infty$, $\sigma >0$ and $s\in [0,\sigma)$. Then, it holds for all $u\in L^{p_0} \cap \dot{H}^\sigma_{p_1}$
$$ \|u\|_{\dot{H}^{s}_p}\lesssim \|u\|_{L^{p_0}}^{1-\theta}\,\, \|u\|_{\dot{H}^{\sigma}_{p_1}}^\theta, $$
where $\theta=\theta_{s,\sigma}(p,p_0,p_1)=\frac{\frac{1}{p_0}-\frac{1}{p}+\frac{s}{n}}{\frac{1}{p_0}-\frac{1}{p_1}+\frac{\sigma}{n}}$ and $\frac{s}{\sigma}\leq \theta\leq 1$.
\end{proposition} 

\begin{proposition}[Fractional Leibniz rule, \cite{ReissigEbert2018}] \label{FractionalLeibniz}
Let us assume $s>0$, $1\leq r \leq \infty$ and $1<p_1,\, p_2,\, q_1,\, q_2 \le \infty$ satisfying the relation
\[ \frac{1}{r}=\frac{1}{p_1}+\frac{1}{p_2}=\frac{1}{q_1}+\frac{1}{q_2}.\]
Then, it holds for any $u\in \dot{H}^s_{p_1} \cap L^{q_1}$ and $v\in \dot{H}^s_{q_2} \cap L^{p_2}$
$$\big\||D|^s(u \,v)\big\|_{L^r}\lesssim \big\||D|^s u\big\|_{L^{p_1}}\, \|v\|_{L^{p_2}}+\|u\|_{L^{q_1}}\, \big\||D|^s v\big\|_{L^{q_2}}. $$    
\end{proposition}

\begin{proposition}[Fractional powers, \cite{ReissigEbert2018}] \label{FractionalPowers}
Let $p>1$, $1< r <\infty$ and $u \in H^{s}_r$, where $s \in \big(\frac{n}{r},p\big)$. Let us denote by $F(u)$ one of the functions $|u|^p,\, \pm |u|^{p-1}u$. Then, the following estimates hold:
$$\|F(u)\|_{H^{s}_r}\lesssim \|u\|_{H^{s}_r}\,\, \|u\|_{L^\infty}^{p-1} \quad \text{ and }\quad \| F(u)\|_{\dot{H}^{s}_r}\lesssim \|u\|_{\dot{H}^{s}_r}\,\, \|u\|_{L^\infty}^{p-1}. $$
\end{proposition}

\begin{proposition}[A fractional Sobolev embedding, \cite{DaoReissig1}] \label{Embedding}
Let $0< s_1< \frac{n}{2}< s_2$. Then, for any function $u \in \dot{H}^{s_1} \cap \dot{H}^{s_2}$ we have
$$ \|u\|_{L^\ity} \lesssim \|u\|_{\dot{H}^{s_1}}+ \|u\|_{\dot{H}^{s_2}}. $$
\end{proposition}

%=====================================================================
\section*{Acknowledgments}
This research was partly supported by Vietnam Ministry of Education and Training under grant number B2023-BKA-06.

%=================================================================================={References}


\begin{thebibliography}{00}
\bibliographystyle{plain}
\bibitem{Matsumura76} A. Matsumura, On the asymptotic behavior of solutions of semi-linear wave equations, \textit{Publ. Res. Inst. Math. Sci.}, \textbf{12} (1976), 169-189.

\bibitem{Matsumura77} A. Matsumura, Energy decay of solutions of dissipative wave equations, \textit{Proc. Japan Acad. Ser. A Math. Sci.}, \textbf{53} (1977), 232-236.

\bibitem{NarazakiReissig2013} T. Narazaki, M. Reissig, $L^1$ estimates for oscillating integrals related to structural damped wave models, in: M. Cicognani, F. Colombini, D. Del Santo (Eds.), Studies in Phase Space Analysis with Applications to PDEs, \textit{Progr. Nonlinear Differential Equations Appl., Birkh\"auser}, (2013), 215-258.

\bibitem{DabbiccoReissig2014} M. D'Abbicco, M. Reissig, Semilinear structural damped waves, \textit{Math. Methods Appl. Sci.}, \textbf{37} (2014), 1570-1592.

\bibitem{Shibata2000} Y. Shibata, On the rate of decay of solutions to linear viscoelastic equation, \textit{Math. Methods Appl. Sci.}, \textbf{23} (2000), 203-226.

\bibitem{DaoReissig1} T.A. Dao, M. Reissig, An application of $L^1$ estimates for oscillating integrals to parabolic like semi-linear structurally damped $\sigma$-evolution models, \textit{J. Math. Anal. Appl.}, \textbf{476} (2019), 426-463.

\bibitem{DaoReissig2} T.A. Dao, M. Reissig, $L^1$ estimates for oscillating integrals and their applications to semi-linear models with $\sigma$-evolution like structural damping, \textit{Discrete Contin. Dyn. Syst. A.}, \textbf{39} (2019), 5431-5463.

\bibitem{DuongKainaneReissig2015} P.T. Duong, M.M. Kainane, M. Reissig, Global existence for semi-linear structurally damped $\sigma$-evolution models, \textit{J. Math. Anal. Appl.}, \textbf{431} (2015), 569-596.

\bibitem{DabbiccoEbert2017} M. D'Abbicco, M.R. Ebert, A new phenomenon in the critical exponent for structurally damped semi-linear evolution equations, \textit{Nonlinear Anal.}, \textbf{149} (2017), 1-40.

\bibitem{DabbiccoEbert2021} M. D'Abbicco, M.R. Ebert, $L^p-L^q$ estimates for a parameter-dependent multiplier with oscillatory and diffusive components, \textit{J. Math. Anal. Appl.}, \textbf{504} (2021), Paper No. 125393, 28 pp.

\bibitem{ReissigEbert2018} M.R. Ebert, M. Reissig, ``Methods for partial differential equations, qualitative properties of solutions, phase space analysis, semilinear models'', Birkh\"auser, 2018.

\bibitem{IkehataTakeda2019} R. Ikehata, H. Takeda, Asymptotic profiles of solutions for structural damped wave equations, \textit{J. Dyn. Differ. Equ.}, \textbf{31} (2019), 537--571.

\bibitem{DaoNga2024} T.A. Dao, B.T. Nga, Some remarks on the asymptotic profile of solutions to structurally damped $\sigma$-evolution equations, \textit{preprint}.

\bibitem{Nishihara2003} K. Nishihara, $L^p$-$L^q$ estimates of solutions to the damped wave equation in 3-dimensional space and their application, \textit{Math. Z.}, \textbf{244} (2003), 631--649.

\bibitem{Narazaki2004} T. Narazaki, $L^p$-$L^q$ estimates for damped wave equations and their applications to semi-linear problem, \textit{J. Math. Soc. Jpn.}, \textbf{56} (2004), 585--626.

\bibitem{TodorovaYordanov2001} G. Todorova, B. Yordanov, Critical exponent for a nonlinear wave equation with damping, \textit{J. Differential Equations}, \textbf{174} (2001), 464--489.

\bibitem{Zhang2001} Q.S. Zhang, A blow-up result for a nonlinear wave equation with damping: the critical case, \textit{C. R. Acad. Sci. Paris S\'er. I Math.}, \textbf{333} (2001), 109--114.

\bibitem{IkehataTanizawa2005} R. Ikehata, K. Tanizawa, Global existence of solutions for semilinear damped wave equations in $\mathbf{R}^N$ with noncompactly supported initial data, \textit{Nonlinear Anal.}, \textbf{61} (2005), 1189--1208.

\bibitem{IkedaOgawa2016} M. Ikeda, T. Ogawa, Lifespan of solutions to the damped wave equation with a critical nonlinearity, \textit{J. Differential Equations}, \textbf{261} (2016), 1880--1903.

\bibitem{FujiwaraIkedaWakasugi2021} K. Fujiwara, M. Ikeda, Y. Wakasugi, On the Cauchy problem for a class of semilinear second order evolution equations with fractional Laplacian and damping, \textit{Nonlinear Differ. Equ. Appl.}, \textit{28} (2021), no. 63.

\bibitem{DuongReissig2017} P.T. Duong, M. Reissig, The external damping Cauchy problems with general powers of the Laplacian, \textit{New Trends in Analysis and Interdisciplinary Applications, Trends in Mathematics}, (2017), 537--543.

\bibitem{Duong2019} P.T. Duong, Some results on the global solvability for structurally damped models with special nonlinearity, \textit{Ukrainian Math. J.}, \textbf{70} (2019), 1395--1418.

\bibitem{DaoReissig2021} T.A. Dao, M. Reissig, Blow-up results for semi-linear structurally damped $\sigma$-evolution equations, In: M. Cicognani, D. Del Santo, A. Parmeggiani, M. Reissig (eds.), \emph{Anomalies in Partial Differential Equations}, Springer INdAM Series, \textbf{43} (2021), 213--245.

\bibitem{IkedaWakasugi2015} M. Ikeda, Y. Wakasugi, A note on the lifespan of solutions to the semilinear damped wave equation,
\textit{Proc. Amer. Math. Soc.}, \textbf{143} (2015), 163--171.

\bibitem{IkehataTakeda2017} R. Ikehata and H. Takeda, Critical exponent for nonlinear wave equations with frictional and viscoelastic damping terms, \textit{Nonlinear Anal.}, \textbf{148} (2017), 228-253.

\bibitem{DAbbicco2017} M. D'Abbicco, $L^1-L^1$ estimates for a doubly dissipative semilinear wave equation, \textit{Nonlinear Differ. Equ. Appl.}, \textbf{24} (2017), 1-23.

\bibitem{DaoMichihisa2020} T.A. Dao, H. Michihisa, Study of semi-linear $\sigma$-evolution equations with frictional and visco-elastic damping, \textit{Commun. Pure Appl. Anal.}, \textbf{19} (2020), no. 3, 1581-1608.

\bibitem{MezadekMezadekReissig2020}
M.K. Mezadek, M.K. Mezadek, M. Reissig, Semilinear wave models with friction and viscoelastic damping, \textit{Math. Meth. Appl. Sci.}, \textbf{43} 2020, 3117-3147.

\bibitem{MezadekMezadekReissig2022}
M.K. Mezadek, M.K. Mezadek, M. Reissig, Semilinear $\sigma$-evolution models with friction and visco-elastic type damping, \textit{Nonlinear Differ. Equ. Appl.}, \textbf{29}(2022), no. 66.

\bibitem{IkehataSawada2016} R. Ikehata and A. Sawada, Asymptotic profile of solutions for wave equations with frictional and viscoelastic damping terms, \textit{Asymptot. Anal.}, \textbf{98} (2016), 59-77.

\bibitem{IkehataMichihisa2019} R. Ikehata, H. Michihisa, Moment conditions and lower bounds in expanding solutions of wave equations with double damping terms, \textit{Asymptot. Anal.}, \textbf{114} (2019), 19--36.

\bibitem{ChenDAbbiccoGirardi2022} W. Chen, M. D'Abbicco, G. Girardi, Global small data solutions for semilinear waves with two dissipative terms, \textit{Ann. Mat. Pura. Appl.}, \textbf{201} (2022), 529-560.
\end{thebibliography}
\end{document}